\documentclass{amsart}
\usepackage{amssymb,bbm}
\usepackage{mathtools}
\usepackage{color}
\usepackage{url}
\usepackage{verbatim}

\def\N{{\mathbb N}}
\def\R{{\mathbb R}}
\def\P{{\mathbb P}}
\def\E{{\mathbb E}}

\def\eps{\varepsilon}
\def\cal{\mathcal}
\newcommand{\diff}{\mathop{}\mathopen{}\mathrm{d}}
\newcommand{\card}{\mathop{}\mathopen{}\mathrm{card}}
\newcommand\ind[1]{\mathbbm{1}_{\left\{#1\right\}}}
\newcommand\croc[1]{\left\langle #1\right\rangle}
\newcommand\steq[1]{\stackrel{\text{\rm #1.}}{=}}

\newtheorem{theorem}{Theorem}[section]

\newtheorem{proposition}[theorem]{Proposition}
\newtheorem{corollary}[theorem]{Corollary}
\theoremstyle{definition}
\newtheorem{definition}[theorem]{Definition}

\theoremstyle{remark}
\newtheorem{remark}{Remark}

\title[The Equilibrium States of Large Networks]{The Equilibrium States\\ of Large Networks of Erlang Queues}
\date{\today}
\address[D. Martirosyan, Ph. Robert]{INRIA Paris, 2 rue Simone Iff, F-75012 Paris, France}
\author{Davit Martirosyan}\thanks{This work has been supported by the Celtic Plus project SENDATE TANDEM (C2015/3-2)}
\email{Martirosyan.Davit@gmail.com}
\author{Philippe  Robert}
\email{Philippe.Robert@inria.fr}
\urladdr{http://team.inria.fr/rap/robert}

\begin{document}

\begin{abstract}
The equilibrium properties of allocation algorithms for networks  with a large number of nodes with finite capacity are investigated.  Every node  is receiving a flow of requests and  when a request arrives at a saturated node, i.e. a node whose capacity is fully utilized, an allocation algorithm may attempt to re-allocate the request to a non-saturated node.  For the algorithms considered, the re-allocation comes at a price: either an extra-capacity  is required in the system or the processing time of a re-allocated request is increased. The paper analyzes the properties of the equilibrium points of the asymptotic associated dynamical system when the number of nodes gets large. At this occasion the classical model of {\em Gibbens, Hunt and Kelly} (1990) in this domain is revisited. The absence of known Lyapunov functions for the corresponding dynamical system complicates significantly the analysis.  Several techniques are used:  Analytic and scaling methods to identify the equilibrium points. We identify the subset of parameters for which the limiting stochastic model of these networks  has multiple equilibrium points.  Probabilistic approaches, like coupling, are used to prove the stability of some of them. A criterion  of exponential stability with the spectral gap of the associated linear operator of equilibrium points is also obtained. 
\end{abstract}

\maketitle

 \vspace{-5mm}

\bigskip

\hrule

\vspace{-3mm}

\tableofcontents

\vspace{-1cm}

\hrule

\bigskip

\section{Introduction}
In this paper we study the time evolution properties of  large stochastic networks with  finite capacity nodes.  Each node of these networks receives a flow of jobs, it has a maximal number of requests which can be present at the same time, it is the capacity of the queue. If a job is accepted, it is getting served immediately upon arrival. In the following, the {\em saturation} of a node  will refer to the fact that its current number of requests is maximal. If a request arrives at a saturated node, i.e.\ with no place left to be accommodated, it may be rejected or allocated to another node according to some allocation algorithm. In this case, by borrowing the terminology of communication networks, it is said that the request is {\em re{-}routed}.

We study two  classes of re-routing algorithms. For both of them, the re-routing of a request comes at a price for the network, either with a larger capacity required or with a longer processing/sojourn time. They are defined as follows. If a request accepted at its arrival node, it  is processed at rate $\mu_1{>}0$. Otherwise, if a request cannot be accommodated at its arrival node:
\begin{enumerate}
\item {\bf The Routing with Increased Sojourn Time (RIST) Algorithm.}\\
A non-saturated node is chosen at random to accommodate the request  which  is processed at rate $\mu_2$ with  $0{<}\mu_2{\le}\mu_1$. A rerouted job stays, on average,  longer in the network for this algorithm.   If all nodes are saturated, the request is rejected. 

This type of model  is used to take into account the fact that, in some contexts, the transfer time of a rerouted job is not negligible.  A special case has already been analyzed in Malyshev and Robert~\cite{Malyshev} when the capacity of each node is $1$. See also Remark~2.1 of Tibi~\cite{Tibi}.

A variant of this algorithm is also discussed in Section~\ref{RiptSec}: When an arriving job finds a node saturated, it picks another node at random, again and again until it finds a non-saturated node, provided it makes less than $p_0$ attempts, otherwise it is rejected. The RIST algorithm corresponds to the case $p_0{=}{+}\infty$. 

\medskip

\item {\bf The Dynamic Alternative Routing (DAR) Algorithm},\\ {\em Gibbens, Hunt and Kelly} (1990).\\
Two other nodes are chosen at random. If both of them are non-saturated, the request takes one place in each of them. Otherwise, the request is rejected. This algorithm has been initially considered by Gibbens et al.~\cite{Gibbens} in 1990 to cope with congestion in communication networks.  The nodes are links of the network and connections are established on links. When a connection requires an already  saturated link $(AB)$ connecting two vertices $A$ and $B$, the algorithm attempts to establish the connection between $A$ and $B$ by taking another vertex $C$ at random and by using  the two links  $(AC)$ and $(CB)$.  See also Kelly~\cite{Kelly} and Marbukh~\cite{Marbukh}.

\end{enumerate}
The main goal of the mathematical studies  of these networks is of quantifying the benefit of rerouting mechanisms. To determine if it is worthwhile  to design routing algorithms rather than doing nothing, i.e.\ rejecting right away  jobs  arriving at saturated nodes. For this purpose,  the probability  that, at equilibrium,  a request is 
\begin{itemize}
\item[a)]  accepted without re-routing;
\item[b)]  rejected, i.e.\ that it cannot be accommodated even by re-routing
\end{itemize}
are the main quantities of interest. 

We first have a non-formal presentation of the problems associated to these algorithms.  The main problem with re-routing is the following. If there is a significant number of saturated nodes, then an important fraction of the resources of the network (capacity, processing time) will be consumed by the re-routed jobs, making re-routing more likely, to the detriment of the criterion associated to a). Furthermore, if there are too many saturated nodes, the loss rate may even be non-negligible, affecting the criterion associated to b).

For the DAR algorithm it has been shown  by Gibbens et al.~\cite{Gibbens}, through some approximations and  numerical experiments, that  these algorithms  exhibit in some cases an unpleasant property. It may happen that the network can stay for a very long amount of time in different regimes (set of states): one where most of requests/jobs are accepted without re-routing and other ones for which a significant fraction of jobs are re-routed. This is a metastability property which is well known in statistical physics. Roughly speaking, there are multiple stable sets of states and the switching time between them is ``large''. Though this property is closely linked to the existence of multiple equilibrium points, it will not be discussed in this paper. See den Hollander~\cite{Hollander}, Bovier and den  Hollander~\cite{Bovier} and Olivieri et al.~\cite{Olivieri} for example.

The paper Gibbens et al.~\cite{Gibbens} in 1990  has attracted a lot of attention, mainly because of the  original stability properties that were suggested in this study, at least in a stochastic network context. See also Marbukh~\cite{Marbukh}. It had a strong impact in the sense that it stressed the undesirable phenomena that can happen without some care in the design of allocation algorithms. Nevertheless, outside the mean-field result of Graham and M\'el\'eard~\cite{GM} in 1993,  there have been few rigorous mathematical results on this important class of models since the appearance of that paper. See Section~4.3 of Kelly~\cite{Kelly}.

We introduce the mathematical framework used to study these two classes of algorithms. 

\subsection*{Mathematical Context: Mean-Field Convergence}
It is assumed that the requests arrive  at each of the $N$ nodes with finite capacity $C$  according to a Poisson process with rate $\lambda$. The sojourn times of the requests at the node are exponentially distributed, with respective parameters $\mu_1$ and $\mu_2$ for RIST, and $\mu_1$ for DAR. If the state of the $i$th node, $1{\le}i{\le}N$, at time $t{\geq}0$ is given by $Z_i^N(t)$ then, for both classes of algorithms, the process of empirical distribution $(\Lambda^N(t))$, with
\begin{equation}\label{EmpIntro}
\Lambda^N(t)=\frac{1}{N}\sum_{i=1}^N \delta_{Z_i^N(t)},\quad t{\ge}0,
\end{equation}
where $\delta_a$ is the Dirac mass at $a$, has the Markov property. The state space ${\cal X}$ of the process $(Z_i^N(t))$ is finite,
\[
\begin{cases}
  {\cal X}=\{(x,y){\in}\N^2: x{+}y{\le} C\} & {\text{\rm (RIST)}}\\
  {\cal X}=\{0,\ldots,C\}, & {\text{\rm (DAR) }}
 \end{cases}
\]
consequently  $\Lambda^N(t)$, the state of a random node belong to ${\cal P}({\cal X})$,  the set of probability distributions on ${\cal X}$.

Under some mild conditions on the initial state, and with some restrictions for the RIST algorithm, see Section~\ref{RiptSec1}, it can be shown that the sequence of stochastic processes $(\Lambda^N(t))$ is converging in distribution to a deterministic measure-valued process $(\Lambda(t))$. As a consequence, the propagation of chaos property holds: in the limit, the states of a finite subset of nodes become independent.  See Sznitman~\cite{Sznitman}.  It can be shown that $(\Lambda(t))$ is the solution of a non-linear Fokker-Planck equation, see Frank~\cite{Frank} for example,
\begin{equation}\label{FKP}
\frac{\diff}{\diff t} \Lambda(t)=\Lambda(t) \cdot Q_{\Lambda(t)},
\end{equation}
where, for $m{\in}{\cal P}({\cal X})$, $Q_m$ is the $Q$-matrix of the reversible irreducible Markov process with the invariant distribution $\pi_m$. For example,  for the  DAR algorithm, $Q_m$ is the $Q$-Matrix of an $M/M/C/C$ queue with service rate $\mu_1$ and arrival rate $\lambda h(m(C))$, where $h$ is a quadratic function. See Relation~\eqref{GHK} below. 

We now review the main problems in this context.
\subsection{Existence, Number and Locations of Equilibrium Points.}
An element $m{\in}{\cal P}({\cal X})$ is an equilibrium point of the dynamical system~\eqref{FKP}, if $m{=}\pi_m$ or, equivalently,
\begin{equation}\label{EFKP}
m\cdot Q_m=0. 
\end{equation}
This equation does not have, in general, an explicit solution and, worse, it is even quite difficult to determine the number of these solutions. For example, for the DAR algorithm, the striking observation of Gibbens et al.~\cite{Gibbens}  has shown that, for some specific parameters,  numerical experiments seem to indicate that there may be three solutions.
But, to the best of our knowledge, this statement does not seem to have been established in a more formal way. See  p.~375 of Hunt and Kurtz~\cite{HuntKurtz}.

It should be noted that, in the remarkable experiments of Gibbens et al.~\cite{Gibbens}, the authors have been able to find   convenient  numerical values of ratio of  the average load per node to  capacity  for which the associated dynamical system has three equilibrium points. This is  in fact not that easy, since, as we will prove in Section~\ref{DARsec}, this phenomenon occurs only if this ratio is in an interval of  width $.063$.

Equation~\eqref{EFKP} can be reduced to a  polynomial equation of degree $C$ involving partial sums of the exponential series. This has some (formal) similarities with  the celebrated Erlang fixed point equation for loss networks.  In this case there is also an asymptotic independence property but it is due to a stochastic averaging principle rather than a mean-field convergence. See Kelly~\cite{Kelly}.

For Erlang systems, this part of the study seems to rely more on analytic methods than probabilistic arguments.  This is probably one of the difficulties of these problems: little intuition can be, a priori, extracted from these polynomial equations.  See  Antunes et al.~\cite{Antunes}, Dawson~\cite{Dawson}, Muzychka~\cite{Muzychka} and Rybko and Shlosman~\cite{Rybko} for the analysis of  other ``large'' queueing models.

\subsection{Stability Properties of Equilibrium Points.}
Concerning the properties of equilibrium points of the dynamical system~\eqref{FKP}, there are two overlapping aspects.

\medskip

\noindent {\sc The convergence of $(\Lambda(t))$ to the equilibrium point.}\\
Assuming that Equation~\eqref{EFKP} has a unique solution, i.e.\ the non-linear dynamical system has therefore a unique equilibrium distribution.  The convergence of the dynamical system $(\Lambda(t))$ is in general a challenging issue.  The dynamical system $(\Lambda(t))$ is associated to a non-linear Markov process $(Z(t))$ with values in ${\cal X}$. Due to the time-inhomogeneity of the dynamics, the classical results of convergence of Markov processes cannot be used. 

For a  large class of  examples of non-linear diffusion processes, related to Langevin evolution  equation, there are nevertheless numerous results concerning the rate of  convergence to equilibrium. Furthermore, an exponential decay is proved with explicit bounds for the Wasserstein distance between two solutions starting from different initial states. Several key ingredients are used in this context: some geometric properties, related to curvature,  to prove an exponential decay of the time evolution of the relative entropy with respect to the equilibrium measure and some functional inequalities. See Carrillo et al.~\cite{Carillo} and reference therein. 

For discrete state spaces  but in a time  homogeneous setting,  Caputo et al.~\cite{Caputo} and Dai Pra and Posta~\cite{Daipra} have  adapted some of the  methods of  the diffusion framework to get explicit bounds on the rate of exponential convergence to equilibrium.  Erbar and Maas~\cite{ErbarMaas} and Maas~\cite{Maas}  have recently developed some tools, the analogue of the geometric characteristics used in the diffusive case, to have a general approach to these problems in a discrete state space. Some interesting but specific examples of random walks have been already investigated with these methods, see for example Erbar et al.~\cite{EMT}. Their use in practice, to get explicit constants on the exponential rate of convergence to equilibrium,  are, as it can be expected,  limited for the moment.  In a non-linear setting,  examples are even more rare. See Thai~\cite{Thai} which investigates the case of birth and death processes whose birth and death rates satisfy a convexity relation. 

\noindent {\sc The local stability of an equilibrium.}\\
Given a solution $m_0$ of Equation~\eqref{EFKP}, the stability property of $m_0$ is the fact that if the initial point of the dynamical system $(\Lambda(t))$ defined by Relation~\eqref{FKP} is in a sufficiently small neighborhood of $m_0$, then $(\Lambda(t))$ converges to $m_0$ and, perhaps, exponentially fast. This is also important from the point of view of qualitative properties of the algorithms, since it asserts that the equilibrium point is meaningful. It suggests that there is a set of states where the network will stay ``for some time''.  A  more ambitious goal would be of determining also the {\em basin of attraction} of the stable points, i.e.\ the set of initial states from which they can be reached.  We are, in fact, far from that here.

Unfortunately, even in a simpler setting, when there is a unique equilibrium, it is surprisingly difficult to prove such a stability property for the models of this paper.  The case of the non-linear $M/M/1$ process in Section~\ref{nlMM1} is striking from this point of view. In a finite state space,  if $m_0$ is an equilibrium point and if the  eigenvalues of a Jacobian matrix associated to Relation~\eqref{FKP} at $m_0$ have negative real part then, under some mild regularity conditions, $m_0$ is a locally stable equilibrium point by  Poincar\'e-Lyapunov's Theorem. See Section~7.1 of Verhulst~\cite{Verhulst} for example.  This is our case of course but we have not been able to derive tractable stability results with this criterion.  Furthermore, in an infinite dimensional context, like, for example, the non-linear $M/M/1$ queue analyzed in Section~\ref{DARsec}, there is an additional complication related to  the choice of the norm between probability distributions which has an important impact.  

In the stochastic networks literature, proofs of local stability of specific examples are rare,  Antunes et al.~\cite{Antunes} establishes, through a dimension reduction, a criterion of local stability for equilibrium points of an Erlang network where jobs move from one node to another after completing their services. See also Tibi~\cite{Tibi} and Budhiraja et al.~\cite{Budhiraja} for a discussion on this topic.

\medskip
\noindent {\sc Stability via Lyapunov Functions.}\\
In some cases,  dynamical systems  associated to Relation~\eqref{FKP} may admit a Lyapunov function, i.e.\ a  function $F{:}{\cal P}(\N){\mapsto}\R_+$, such that
\[
\left.\frac{\diff}{\diff t} F(\Lambda(t))\right|_{t=0} < 0, 
\]
for $\Lambda(0){=}m{\in}{\cal P}(\N)$, and  $F(m){\not=}0$. In the case of a unique equilibrium, the existence of such a function may give the desired convergence to equilibrium. An explicit representation of such a function $F$  contains in fact a lot of information on the dynamical system. The state space being the set of probability distributions on a finite or countable set, examples with a Lyapunov function in such a context are quite rare. Antunes et al.~\cite{Antunes} has such a function for Erlang networks with a Jackson-type routing. The Lyapunov function is expressed in terms of a relative entropy with respect to the invariant distribution of some single Erlang queue and a complementary term related to the non-linear dynamics.

Tibi~\cite{Tibi} and Budhiraja et al.~\cite{Budhiraja} have investigated the conditions  under which  a Lyapunov function based on  relative entropy can be constructed. It turns out that, in practice, the  possibilities are in fact limited among ``classical'' models. Both papers mention the fact that their respective (equivalent)  conditions ``{\em local balance}'' or ``{\em local Gibbs}'' for having a Lyapunov function {\em cannot hold} for the DAR algorithm. It  seems to be also the case for the RIST algorithm. This may  partially explain the lack of progress in the mathematical analysis of the DAR algorithm since the original article appeared.

\subsection{Contributions}
The original motivation of this paper could be summarized as follows: In absence of Lyapunov functions, how can we study the equilibrium properties of these algorithms?  We have used a set of quite diverse techniques. For both algorithms  the arrivals of requests are Poisson with rate $\lambda{>}0$ and  their processing times are exponentially distributed. 
\begin{enumerate}
\item {\sc Scaling methods for the number of  equilibrium points.}\\
  For the RIST algorithm with an infinite number of retrials, we are able to determine exactly the set of equilibrium points of the dynamical system associated to the mean-field limit of this system.  See Propositions~\ref{PhiProp} and~\ref{theo1}.

Because of a singularity in the coefficients of the asymptotic dynamical system, a solution  may die in finite time. It turns out that there may be one, two or three equilibria and that one of them is related to this singularity. It corresponds to the case when the network has mostly re-routed requests, see Proposition~\ref{theo1} of Section~\ref{sec1-Sat} for a more precise description. In Section~\ref{Ript1sec}, we show that for the variant of the RIST algorithm with one retrial, the corresponding system of ODEs is in this case without singularity and that there may also be three equilibrium points.

  For the DAR algorithm, by taking the global input rate proportional to the capacity $C$, $\lambda{=}\nu C$ for some $\nu$, it is shown that, with a convenient scaling, as $C$ gets large, the dynamical system converges to a system of ODEs. This asymptotic dynamical system can be expressed in terms of a mean-field limit of a network of interacting  $M/M/1$ queues with infinite capacity. It can be described in terms of a non-linear $M/M/1$ queue.  In this case too, there is an equilibrium regime with an intuitive interpretation as an underloaded regime. See Section~\ref{2StabSec}.

With this scaling analysis of the equilibrium equations, we can give the  conditions on the parameters of the network so that for a fixed, but sufficiently  large capacity $C$, there are three equilibrium points, as suggested and conjectured by Gibbens et al.\cite{Gibbens} but, to the best of our knowledge, not proved until now. Additionally, the limiting values of these equilibrium points when $C$ goes to infinity are identified. See Theorem~\ref{ThNbFP}.
\medskip
\item {\sc Probabilistic approach to prove stability.}\\
We show the stability of some of the equilibrium points associated to congested regimes of the RIST and DAR algorithm. Coupling methods are used, by constructing an ad-hoc order relation and by deriving several technical estimates. See Propositions~\ref{theo1} and~\ref{Stab2prop}. 
\medskip
\item {\sc Spectral approach to exponential stability.}\\
For the other equilibrium points, the stability problem is much harder.   They are defined with the help of solutions of some polynomial equation without much insight on the significance of them.  This is where, when available,  a Lyapunov function is useful.  We have tried to derive a stability criterion in terms of the spectral gap $\kappa$ of the linear Markov process associated to the equilibrium point. This quantity is known for some of these classical models like the Erlang queue or the $M/M/1$ queue. See Chen~\cite{Chen} and van Doorn~\cite{vanDoorn} for example. The (rough) idea is that if the rate of convergence to equilibrium is sufficiently large, then the non-linear perturbations will not take the trajectory away from the neighborhood of the equilibrium point. Unfortunately this intuitive picture is difficult to establish rigorously. See Theorem~\ref{theoNLMM1} for an example of such a situation. By using an $L_2$-norm and with several estimates, we have nevertheless  been able to establish some exponential stability criteria in terms of the spectral gap for the RIST algorithm and for the non-linear $M/M/1$ queueing model of the DAR algorithm. This is a surprisingly difficult problem, even when the dynamical system has a unique fixed point.  A good example is the non-linear $M/M/1$ queue of Section~\ref{nlMM1} with $\nu{>}1$. Only  a partial result for this case has been obtained.
\end{enumerate}

\subsection*{Addendum}
During the writing of this paper, in August 2018, we have learned the sad news of the death of Richard Gibbens.  As it is plainly clear, his remarkable paper with his colleagues in 1990 has been the main motivation of this work. We would like to pay tribute  to his memory. 

\section{The Routing with Increased Sojourn Time (RIST) Algorithm}\label{RiptSec}
For $i{\in}\{1,\ldots,N\}$, requests arrive at node~$i$ with capacity $C$ according to a Poisson process with rate $\lambda$. If node~$i$ is not full at one of this instants, the corresponding request is accepted and its sojourn time is exponentially distributed with rate $\mu_1$, it will be referred to as a class~$1$ customer. If node~$i$ is full, another node is picked at random. If this node is not saturated, the request is allocated at this node. Otherwise, another node is picked at random, the maximum number of attempts is limited to $p_0{\in}\N{\cup}\{{+}\infty\}$. Otherwise, the request is rejected. We will mostly investigate the case $p_0{=}{+}\infty$. The sojourn time of rerouted requests is exponentially distributed with rate $\mu_2{<}\mu_1$,  they are defined as class~$2$ customers.
\subsection{The Associated Dynamical System}\label{RiptSec1}
\addcontentsline{toc}{section}{\thesubsection \hspace{3.5mm} The Associated Dynamical System}

We introduce some notations used throughout this section.  The state space of a node  is given by ${\cal X}$ with
\begin{equation}\label{calX}
  \begin{cases}
    {\cal X}{=}\{z{=}(x,y){\in}\N^2\,|\, x{+}y{\leq} C\},\\
    {\cal X}_+{=}\{z{=}(x,y){\in}\N^2\,|\, x{+}y{<} C\}  \text{ and }{\cal X}_+^c{=}{\cal X}{\setminus}{\cal X}_+{=}\{z{=}(x,y){\in}\N^2\,|\, x{+}y{=} C\} .
  \end{cases}
\end{equation}
If $z{=}(x,y){\in}{\cal X}$, $x$  [resp. $y$] is the number of class~1 [resp.~2] customers.

The set of probability distributions on ${\cal X}$ is denoted by ${\cal P}({\cal X})$ and  ${\cal C}(\R_+,{\cal P}({\cal X}))$ is the set of continuous functions with values in ${\cal P}({\cal X})$. If $\zeta{\in}{\cal P}({\cal X})$ and $f$ is a real-valued function on ${\cal X}$ and $A{\subset}{\cal X}$,  we denote
\[
\croc{\zeta,f}{=}\int f(z)\, \zeta(\diff z)\text{ \ and\  } \zeta(A){=}\int_A \zeta(\diff z).
\]

The state space of the process describing the whole network is  ${\cal S}_N{\steq{def}}{\cal X}^N$. For $t{\ge}0$ and $i{\in}\{1,\ldots,N\}$, $(X_i^N(t))$ [resp.  $(Y_i^N(t))$ ] denotes the number of class~$1$ [resp.~$2$]  customers at node $i$. 

For $\mu{\in}\{\mu_1,\mu_2\}$, $({\cal N}_\mu^{ij})$ is an i.i.d sequence of Poisson processes with rate $\mu$. They associated to the service times of class $1$ and $2$ customers at the nodes of the networks. For $i$, $j{\ge}$ ${\cal N}_\mu^{ij}$ is for the service times of $j$th server of the $i$th node. Similarly, $({\cal N}_\lambda^i){=}((t_n^i))$  is an i.i.d. sequence of Poisson processes with rate $\lambda$. For $i{\in}\N$, ${\cal N}_\lambda^i$ is the arrival process at node $i$. All Poisson processes are assumed to be independent. Additionally $(U_{nk}^i)$ is an i.i.d. sequence of uniform random variables on $[0,1]$,  also referred to as ``marks'' of the point process $\overline{\cal N}_\lambda^i$ defined by
  \[
  \overline{\cal N}_\lambda^i\steq{def} \sum_{n\in\N} \delta_{(t_n^i,(U_{nj}^i, j\ge 1))},
  \]
  $\overline{\cal N}_\lambda^i$ is a {\em  marked Poisson point process}. See Chapter~5 of Kingman~\cite{Kingman}. If ${\cal M}{\steq{def}}[0,1]^\N$ is the space of marks, then $ {\cal N}_\lambda^i(\diff t){=} \overline{\cal N}_\lambda^i(\diff t,{\cal M})$ is the arrival process at node~$i$.  

A mark is associated to an arrival instant, it is used to determine to which queue a customer goes if the node where it arrives is saturated. For node $i$, this is (formally) done with a functional $T_i^N$ on ${\cal S}{\times}{\cal M}$ defined as follows.  For $(z,u){=}((x_i,y_i),(u_j)){\in}{\cal S}{\times}{\cal M}$, define $(v_j(u)){=}(\lfloor N u_j\rfloor)$,  with $0$ identified to $N$. For $i$ and $n{\in}\N$, $(v_j(U_{nj}^i))$ is an i.i.d sequence uniformly distributed on $\{1,\ldots,N\}$.  With this notation, we take
\begin{equation}\label{defT}
  T_i^N(z,u){=}
  \begin{cases}
    i &\text{ if } x_i{+}y_i{<}C,\\
    v_k(u) &\text{ otherwise, if } k{=}\inf\{j\mid x_{v_j(u)}{+}y_{v_j(u)}{<}C\}{\le}p_0,\\
    {+}\infty &\text{ otherwise }.
  \end{cases}
\end{equation}
The variable $T_i^N(z,u)$ will be used in the following only when node $i$ is saturated. 

When the system is in a state $z{\in}{\cal S}_N$ with  $x_i{+}y_i{=}C$, for $i{\in}\{1,\ldots,N\}$, if there is an arrival  with a mark $u$  at node~$i$, then, provided that   $T_i^N(z,u)$ is less than $p_0$, the queue with index  $T_i^N(z,u)$  receives the  re-routed customer. When $p_0{=}{+}\infty$,  if $x_j{+}y_j{<} C$ for some $j{\in}\{1,\ldots,N\}$, then it is easy to check that
\[
\P(T_i(z,U){=}j)={1}{\bigg/}\sum_{k=1}^N \mathbbm{1}_{{\cal X}_+}(x_k,y_k),
\]
where ${\cal X}_+$ is defined by Relation~\eqref{calX}.

If the addition of marks looks somewhat formal, it has the advantage of giving a neat framework to handle the stochastic calculus associated to the evolution equations of the system. The martingale property will be with respect to the filtration $({\cal F}_t)$ defined by, for $t{\ge}0$,
\[
{\cal F}_t=\left.\sigma\left\{\overline{\cal N}_\lambda^i([0,s]{\times} B), {\cal N}_{\mu_1}^{ij}([0,s]), {\cal N}_{\mu_2}^{ij}([0,s])\right| i, j{\in}\N, s{\le}t, B{\in}{\cal B}({\cal M})\right\}.
\]
\subsubsection*{Evolution Equations}
We assume from now on that $p_0{=}{+}\infty$, a customer is rejected only if all nodes are saturated. We consider briefly the case of a finite $p_0$ in Proposition~\ref{DynSysProp2} and in Section~\ref{Ript1sec}.

The state of the system is represented by a process $(Z^N(t)){=}((X_i^N(t),Y_i^N(t))$ which is c\`adl\`ag and  satisfies the following stochastic differential equations, for $1{\le}i{\le}N$,  
\begin{align}
  \diff X_i^N(t) &= \ind{(X_i^N(t-),Y_i^N(t-)){\in}{\cal X}_+}{\cal N}_{\lambda}^i(\diff t)-\sum_{\ell=1}^{X_i^N(t-)} {\cal N}_{\mu_1}^{i,\ell}(\diff t), \label{SDEX}\\
  \diff Y_i^N(t) &= \sum_{j=1}^N \ind{\substack{(X_j^N(t-),Y_j^N(t{-})){\in}{\cal X}_+^c\\T_j(u,Z^N(t-))=i}}\overline{{\cal N}}_{\lambda}^{j}(\diff t, \diff u)\label{SDEY}\\
&\hspace{5cm}  -\sum_{\ell=1}^{Y_i^N(t-)} {\cal N}_{\mu_2}^{i,\ell}(\diff t),\notag
\end{align}
where $U(t{-})$ denotes the left limit of $U$ at $t{>}0$.

The first term of the right-hand side of Relation~\eqref{SDEY} corresponds to arrivals finding their arrival node saturated and are allocated to some non-saturated node by repeated random sampling of  nodes until one of them can accommodate it. Note that it is not excluded that all nodes are saturated, in this case the customer is rejected since it is not allocated anywhere.
\subsubsection*{Martingales}
We recall classical results on the martingales associated to marked Poisson point processes. See Sections~4.4 and~4.5 of Jacobsen~\cite{Jacobsen} for example. See also Last and Brandt~\cite{Last}.
If $h$ is a  bounded function on $\R_+{\times}\{1,\ldots,N,{+}\infty\}$, such that $h(\cdot,{+}\infty){\equiv}0$ and $\Lambda^N$ is the empirical distribution defined by Relation~\eqref{EmpIntro}, the process 
\begin{multline*}
\left(\int_{[0,t]\times{\cal M}}\hspace{-5mm} h\left(Z^N(s{-}),T_i^N(Z^N(s{-}),u)\right)\overline{\cal N}_\lambda^i(\diff s, \diff u)\right. \\ \left.{-}\lambda\sum_{j=1}^N \int_0^t \frac{h(Z^N(s),j)}{N\Lambda^N(s)({\cal X}_+)}\mathbbm{1}_{{\cal X}_+}\left(X_j^N(s),Y_j^N(s)\right)\diff s \right)
\end{multline*}
is a martingale whose previsible increasing process (quadratic variation) is given by
\[
\left(\lambda\sum_{j=1}^N \int_0^t \frac{h^2(Z^N(s),j)}{N\Lambda^N(s)({\cal X}_+)}\mathbbm{1}_{{\cal X}_+}\left(X_j^N(s),Y_j^N(s)\right)\diff s \right)
\]
For example, the integration of  SDE~\eqref{SDEY} and the compensation of the Poisson processes give the relation
\begin{multline*}
  Y_i^N(t) =  Y_i^N(0) +M_i^N(t) \\
  +\lambda \sum_{j=1}^N \int_0^t  \ind{\substack{(X_j^N(s),Y_j^N(s)){\in}{\cal X}_+^c\\(X_i^N(s),Y_i^N(s)){\in}{\cal X}_+}}\frac{1}{N\Lambda^N(s)({\cal X}_+)}\,\diff s -\mu_2\int_0^t Y_i^N(s) \,\diff s,
\end{multline*}
which can be written under the more compact form
\begin{multline}\label{eqY}
  Y_i^N(t) =   Y_i^N(0) +M_i^N(t)\\
    + \lambda\int_0^t \ind{(X_i^N(s),Y_i^N(s)){\in}{\cal X}_+} \frac{\Lambda^N(s)({\cal X}_+^c)}{\Lambda^N(s)({\cal X}_+)}\,\diff s -\mu_2\int_0^t Y_i^N(s) \,\diff s,
\end{multline}
where $(M_i^N(t))$ is a martingale whose previsible increasing process is
\begin{multline}\label{eqcrocY}
\left(\croc{M_i^N}(t)\right)\\=\left( \lambda\int_0^t \ind{(X_i^N(s),Y_i^N(s)){\in}{\cal X}_+} \frac{\Lambda^N(s)({\cal X}_+^c)}{\Lambda^N(s)({\cal X}_+)}\,\diff s+\mu_2\int_0^t Y_i^N(s) \,\diff s,\right).
\end{multline}
\subsubsection*{Empirical Distributions}
The empirical distribution process $(\Lambda^N(t))$  associated to $(Z^N(t))$ is defined by, for $t{\ge}0$,
\begin{equation}\label{Emp}
\Lambda^N(t)(f){=}\int_{{\cal X}} f(z)\Lambda^N(t)(\diff z) \steq{def}\frac{1}{N}\sum_{i=1}^N f\left(X_i^N(t),Y_i^N(t)\right),
\end{equation}
for any non-negative function $f$ on ${\cal X}$. It is a stochastic process with values in the set  ${\cal P}({\cal X})$ of probability distributions on ${\cal X}$.

As in the derivation of Relation~\eqref{eqY}, the integration of Equations~\eqref{SDEX} and~\eqref{SDEY} and the compensation of Poisson processes give the relation
\begin{multline}\label{EmpN}
  \croc{\Lambda^N(t),f}=   \croc{\Lambda^N(0),f}+M_f^N(t)\\
  +\lambda\int_0^t\int_{{\cal X}_+}  \nabla_1^+(f)(z)\,\Lambda^N(s)(\diff z)\,\diff s-\mu_1\int_0^t\int_{z{=}(x,y){\in}{\cal X}} x\nabla_1^-(f)(z)\,\Lambda^N(s)(\diff z)\,\diff s\\
  + \lambda\int_0^t \int_{{\cal X}_+} \nabla_2^+(f)(z)\frac{\Lambda^N(s)({\cal X}_+^c)}{\Lambda^N(s)({\cal X}_+)}\,\Lambda^N(s)(\diff z)\,\diff s\\
-\mu_2\int_0^t \int_{z{=}(x,y){\in}{\cal X}} y\nabla_2^-(f)(z)\,\Lambda^N(s)(\diff z)\,\diff s,
\end{multline}
for a real-valued function $f$ on ${\cal X}_+$, where, if $z{=}(x,y)$,  $\nabla_1^\pm(f)(z){=}f(x{\pm}1,y){-}f(x,y)$ and $\nabla_2^\pm(f)(z){=}f(x,y{\pm}1){-}f(x,y)$ are the gradient operators, and $(M_f^N(t))$ is a martingale. 

\begin{proposition}[Dynamical System with an Unbounded Number of Retrials]\label{DynSysProp}
If $p_0{=}{+}\infty$ and  $(\Lambda(t))$ is the unique solution of the differential equation
\begin{multline}\label{DynSys}
  \frac{\diff}{\diff t}\croc{\Lambda(t),f}= \lambda\croc{\Lambda(t), \nabla_1^+(f)\mathbbm{1}_{{\cal X}_+}}+\mu_1\croc{\Lambda(t),I_1\nabla_1^-(f)}\\
  + \lambda \frac{\Lambda(t)({\cal X}_+^c)}{\Lambda(t)({\cal X}_+)} \croc{\Lambda(t),\nabla_2^+(f)\mathbbm{1}_{{\cal X}_+}} +\mu_2\croc{\Lambda(t),I_2\nabla_2^-(f)},
  \end{multline}
for $t{<}H_0(\Lambda)$, where $I_1(x,y){\steq{def}}x$ and $I_2(x,y){\steq{def}}y$, and
\[
H_0(\zeta)\steq{def}\inf\{t{>}0\mid \zeta(t)({\cal X}_+){=}{0}\},\qquad\text{ for }  (\zeta(t)){\in}{\cal C}(\R_+,{\cal P}({\cal X})),
\]
then, for the convergence in distribution of processes,
\[
\lim_{N\to+\infty} \left(\Lambda^N(t),t{<}H_0\left(\Lambda^N\right)\right)= (\Lambda(t),t{<}H_0(\Lambda)).
\]
\end{proposition}
Note that the sequence of processes $(\Lambda^N(t))$  is in fact a sequence of finite-dimensional processes with dimension $\card({\cal X}_+)$. The convergence in distribution of the proposition refers to the case when the  space of c\`adl\`ag functions with values in $\R_+^{\card({\cal X}_+)}$ is endowed with the uniform norm.

The variable $H_0(\Lambda)$ is the {\em blow-up time} of the dynamical system~\eqref{DynSys}. If finite, it amounts to the fact that the system is completely saturated at the fluid scale. It will be seen in Section~\eqref{sec1-Sat} that the saturated state is a stable equilibrium of the network. Note that, because of its singular aspect,  it cannot be really defined through the ODEs associated to Relation~\eqref{DynSys}. 
\begin{proof}
For the martingale $({M}_f^N(t))$ of Relation~\eqref{EmpN},  with calculations  similar to the ones used in the derivation of Relation~\eqref{eqcrocY}, one gets the existence of a constant $K_T$ such that $\E(\croc{M}_f^N(T)){<}K_T/N$ and, by Doob's Inequality, the convergence in distribution  to $(0)$ of this martingale for the topology associated to uniform convergence on $[0,T]$.

Note that, for $s{<}t$,
  \[
  \left|\int_s^t \int_{{\cal X}_+} \nabla_2^+(f)(z)\frac{\Lambda^N(u)({\cal X}_+^c)}{\Lambda^N(u)({\cal X}_+)}\,\Lambda^N(u)(\diff z)\right|\leq 2\|f\|_{\infty}(t{-}s).
  \]
  By using the criterion of the modulus of continuity and Relation~\eqref{EmpN}, we get that the sequence of processes $(\Lambda^N(t))$ is tight in distribution for the topology of uniform convergence on compact sets. See Theorem~7.3 of Billingsley~\cite{Billingsley} for example.

The solution of Equation~\eqref{DynSys} lives in a finite dimensional state space, the set of probability distributions on ${\cal X}$. The system~\eqref{DynSys} can be seen as a set of ODEs. It has in particular a unique solution up to time $H_0(\Lambda)$.  If
\[
H_\eps(\zeta)=\inf\{t{>}0\mid \zeta(t)({\cal X}_+){>}\eps\}
\]
then $H_\eps(\Lambda^N)$ converges in distribution to $H_\eps(\Lambda)$. By the continuous mapping theorem used in Relation~\eqref{EmpN}, one gets that on the event $\{H_\eps{>}t\}$, the relation
\begin{multline*}
\croc{\Lambda(t),f}=\croc{\Lambda(0),f}{+}\lambda\int_0^t \croc{\Lambda(s), \nabla_1^+(f)\mathbbm{1}_{{\cal X}_+}}\,\diff s{+}\mu_1\int_0^t\croc{\Lambda(s),I_1\nabla_1^-(f)}\,\diff s\\
  + \lambda \int_0^t\frac{\Lambda(s)({\cal X}_+^c)}{\Lambda(s)({\cal X}_+)} \croc{\Lambda(s),\nabla_2^+(f)\mathbbm{1}_{{\cal X}_+}}\,\diff s +\mu_2\int_0^t\croc{\Lambda(s),I_2\nabla_2^-(f)}\,\diff s
\end{multline*}
holds. These equations can be seen as a system of ordinary differential equations, it has clearly a unique solution up to blow-up time $H_0$. The proposition is proved. 
\end{proof}
We state the analogous result when the number of retrials is finite, the notations are the same  as in Proposition~\ref{DynSysProp}. The proof of the proposition being simpler in this case, is skipped. 
\begin{proposition}[Dynamical System with a Maximum of $p_0$ Retrials]\label{DynSysProp2}
If $(\Lambda(t))$ is the unique solution of the differential equations, for $t{>}0$,
\begin{multline}\label{DynSys2}
  \frac{\diff}{\diff t}\croc{\Lambda(t),f}= \lambda\croc{\Lambda(t), \nabla_1^+(f)\mathbbm{1}_{{\cal X}_+}}+\mu_1\croc{\Lambda(t),I_1\nabla_1^-(f)}\\
  + \lambda \Lambda(t)({\cal X}_+^c)\left(\frac{1{-}[\Lambda(t)({\cal X}_+^c)]^{p_0}}{1{-}\Lambda(t)({\cal X}_+^c)}\right)\croc{\Lambda(t),\nabla_2^+(f)\mathbbm{1}_{{\cal X}_+}} +\mu_2\croc{\Lambda(t),I_2\nabla_2^-(f)},
  \end{multline}
then, for the convergence in distribution of processes, the relation
\[
\lim_{N\to+\infty} \left(\Lambda^N(t)\right)= (\Lambda(t))
\]
holds.
\end{proposition}
The non-linear term in the right-hand side of Relation~\eqref{DynSys2} is bounded by $p_0$, and is therefore without singularity.

\subsection{The Number of Equilibrium Points}\label{1LocSec}
\addcontentsline{toc}{section}{\thesubsection \hspace{3.5mm} The Number of Equilibrium Points}
We investigate the fixed points of the linearized version $(\Lambda_R(t))$ of the dynamical system~\eqref{DynSys} when the non-linear term $\Lambda(t)({\cal X}_+^c)$ is replaced by a constant $R{\in}(0,1)$. This dynamical system satisfies the following equations
\begin{multline}\label{LinDynSys}
\croc{\Lambda_R(t),f}{=}\croc{\Lambda_R(0),f}{+}\lambda\int_0^t \croc{\Lambda_R(s), \nabla_1^+(f)\mathbbm{1}_{{\cal X}_+}}\,\diff s\\{+}\mu_1\int_0^t\croc{\Lambda_R(s),I_1\nabla_1^-(f)}\,\diff s
  {+} \lambda \frac{R}{1{-}R}\int_0^t \croc{\Lambda_R(s),\nabla_2^+(f)\mathbbm{1}_{{\cal X}_+}}\,\diff s\\ +\mu_2\int_0^t\croc{\Lambda_R(s),I_2\nabla_2^-(f)}\,\diff s,
\end{multline}
for any real-valued function $f$ on ${\cal X}$, with the notations of Proposition~\ref{DynSysProp}. Recall that for any function $g$  and measure $\mu$ on ${\cal X}$,
\[
\croc{\mu,f}{=}\sum_{(x,y){\cal X}} f(x,y)\mu((x,y))
\]

The process $(\Lambda_R(t))$ describes the evolution of the  law of a classical Erlang model with capacity $C$ where two classes of customers arrive respectively at rate $\lambda$ and $\lambda R/(1{-}R)$ and are served at  rates $\mu_1$ and $\mu_2$. Its invariant distribution on ${\cal X}$ is given by
\begin{equation}\label{piR}
\pi_R(x,y)=\frac{1}{Z_R} \cdot \frac{\rho_1^x}{x!}\frac{\rho_2^y}{y!}\left(\frac{R}{1{-}R}\right)^y,\quad (x,y){\in}{\cal X},
\end{equation}
where $\rho_1{=}\lambda/\mu_1$ and $\rho_2{=}\lambda/\mu_2$ and $Z_R$ is the normalization constant,
\[
Z_R=\sum_{m=0}^C \sum_{x+y=m} \frac{\rho_1^x}{x!}\frac{\rho_2^y}{y!}\left(\frac{R}{1{-}R}\right)^y=\sum_{m=0}^C \frac{1}{m!} \left(\rho_1{+}\rho_2\frac{R}{1{-}R}\right)^m,
\]
and we have
\[
\pi_R\left({\cal X}_+^c\right)=\pi_R((x,y){\in}{\cal X}, x{+}y=C)=\frac{1}{Z_R} \frac{1}{C!} \left(\rho_1{+}\rho_2\frac{R}{1{-}R}\right)^C,
\]
Hence $\pi_R$ defined by~\eqref{piR}  is a fixed point of the dynamical system~\eqref{DynSys} if and only if   $R$ satisfies the relation,
\begin{equation}\label{PiRFix}
R=\pi_R\left({\cal X}_+^c\right).
\end{equation}
The goal of this section is of characterizing completely the solutions of Equation~\eqref{PiRFix}.

Note that Equation~\eqref{PiRFix} can be rewritten as $\Phi_{\rho_1,\rho_2}(z_R){=}0$ with, for $\rho_1$, $\rho_2{>}0$,
\begin{equation}\label{Phi}
\Phi_{\rho_1,\rho_2}(z)\steq{def}\frac{1}{z}\frac{(\rho_1{+}\rho_2 z)^C}{C!}-\sum_{m=0}^{C-1} \frac{ (\rho_1{+}\rho_2 z)^m}{m!},
\end{equation}
and $z_R{=}{R}/{(1{-}R)}$.
The next proposition determines the number of roots of the function $\Phi_{\rho_1,\rho_2}$. It gives all the non-singular equilibrium points of the ODEs~\eqref{DynSys} in the sense. Section~\ref{sec1-Sat} gives a formal presentation of another regime which yields a stable equilibrium when $\rho_2{>}C$.  See the remark below.
\begin{proposition}[Non-Singular Equilibrium Points]\label{PhiProp}
  If $C{\ge}2$ and $\rho_1{<}C$,
  \begin{enumerate}
  \item For $\rho_2{\in}[0,C)$, there is a unique root $z(\rho_1,\rho_2){>}0$ of $\Phi_{\rho_1,\rho_2}$. The function $\rho_2{\to}z(\rho_1,\rho_2)$  is increasing on  $(0,C)$ and if 
    \[
    z(\rho_1,C){\steq{def}}\lim_{\mathclap{\rho_2\nearrow C}}\uparrow z(\rho_1,\rho_2) 
    \]
    then  $z(\rho_1,C)$ is the unique root of $\Phi_{\rho_1,C}$ if $\rho_1{<}C{-}1$ and $z(\rho_1,C){=}{+}\infty$ otherwise. 
\item   For $\rho_2{>}C$, there exists a  non-increasing function $\phi_c{:}(C,{+}\infty){\to}(0,C{-}1)$  such that,
    \[
    \lim_{z\to+\infty} \phi_C(z)=0, \text{ and }     \lim_{z\searrow C} \phi_C(z)=C{-}1,
    \]
    and
      \begin{enumerate}
    \item if $\rho_1{\in}(0,\phi_C(\rho_2))$, $\Phi_{\rho_1,\rho_2}$  has two roots in $\R_+$;
    \item If $\rho_1{\in}(\phi_C(\rho_2),C)$,  $\Phi_{\rho_1,\rho_2}$  does not have any root.
    \end{enumerate}
    In the case $\rho_1{=}\phi_C(\rho_2)$, $\Phi_{\rho_1,\rho_2}$  has a unique root.
  \end{enumerate}
\end{proposition}

\noindent
{\sc Remarks}
\begin{itemize}
\item The location of $\rho_2$ with respect to $C$ and of $\rho_1$ with respect to $\phi_C(\rho_2)$ determines the number of  solutions satisfying Equation~\eqref{PiRFix}.
The function $\phi_C$ can be  in fact (formally)  defined by
\begin{equation}\label{phim}
 \phi_C (\rho_2){\steq{def}}\sup\{\rho_1: m(\rho_1,\rho_2){<}0\}, \text{ with } m(\rho_1,\rho_2){\steq{def}}\min\{\Phi_{\rho_1,\rho_2}(z):z{\ge}0\}.
   \end{equation}
\item As it will be seen in Section~\ref{sec1-Sat}, when $\rho_2{>}C$ there is another, singular, equilibrium which is not mentioned in this proposition, it corresponds to $R{=}1$. Starting from some suitable initial states, the dynamical system~\eqref{DynSys} converges to the Dirac measure $\delta_{(0,C)}$, i.e.\ most of nodes are saturated.

  This situation  corresponds to the case where, in the limit, all requests are rerouted or rejected.  Mathematically, this is a consequence of the possibility that the dynamical system~\eqref{DynSys} may blow-up in finite time, i.e.\  may degenerate. 

\item Note that the case $C{=}1$ analyzed in Malyshev and Robert~\cite{Malyshev} does not exhibit multiple equilibria.
\end{itemize}
We will refer to the different cases of the proposition as, respectively, {\em underloaded},  if $\rho_2{<}C$, {\em critical}, if $\rho_2{=}C$ and {\em overloaded} if $\rho_2{>}C$.
\begin{proof}
Once the results to prove are properly formulated, the proofs of the statements are done via real analysis. First, note that
\begin{equation}\label{eqPhi1}
 \lim_{z\searrow 0}  z\Phi_{\rho_1,\rho_2}(z)=\frac{\rho_1^C}{C!}{>}0 \text{ and } \lim_{z\to+\infty} \frac{\Phi_{\rho_1,\rho_2}(z)}{z^{C-1}}=\frac{\rho_2^{C-1}}{(C{-}1)!}\left(\frac{\rho_2}{C}{-}1\right).
\end{equation}

\subsection*{The underloaded case $\rho_2{\in}(0,C)$}
Define, for $z{\ge}0$,
\[
f_{\rho_1,\rho_2}(z){=}\Phi_{\rho_1,\rho_2}(z)e^{-\rho_2  z}.
\]
After some simple calculations, with telescoping sums,  we get that 
\begin{equation}\label{eqPhi2}
f_{\rho_1,\rho_2}'(z)=  \frac{1}{z^2}\frac{(\rho_1{+}\rho_2 z)^{C-1}}{C!}e^{-\rho_2 z} \left(\rho_2(C{-}\rho_2) z^2+\rho_2(C{-}\rho_1{-}1)z-\rho_1\right).
\end{equation}
The last term of the right{-}hand side of the expression of $f_{\rho_1,\rho_2}'(z)$ is a polynomial of degree $2$. Its value at $z{=}0$ is negative and, since $\rho_2{<}C$, it is converging to ${+}\infty$ when $z$ gets large, hence there exists a unique $z_0{>}0$  such that $f_{\rho_1,\rho_2}'(z_0){=}0$. It necessarily corresponds to a unique extremum of $f_{\rho_1,\rho_2}$ on $\R_+$, a minimum, given the variations of $f_{\rho_1,\rho_2}$, see Relation~\eqref{eqPhi1}.  Hence,  the function $z{\to}f_{\rho_1,\rho_2}(z)$ is  decreasing on  $(0,z_0)$ and increasing on $[z_0,{+}\infty)$ and, since  $f_{\rho_1,\rho_2}$ is converging to $0$ at infinity, the relation $f_{\rho_1,\rho_2}(z){<}0$ holds for $z{\ge}z_0$. Thus, there exists a unique root $z(\rho_1,\rho_2)$ of $f_{\rho_1,\rho_2}$ in $\R_+$, located in $(0,z_0)$. We have therefore the equivalence of the two relations $z{<}z(\rho_1,\rho_2)$ and $\Phi_{\rho_1,\rho_2}(z){>}0$; the latter relation can be expressed as
      \[
z\sum_{m=0}^{C-1}\frac{C!}{m!} \frac{1}{(\rho_1{+}\rho_2 z)^{C-m}}<1.
      \]
      Hence, if $\rho_2'{>}\rho_2$, then $\Phi_{\rho_1,\rho_2'}(z){>}0$ holds if  $z{<}z(\rho_1,\rho_2)$. The function $\rho_2{\to}z(\rho_1,\rho_2)$ is therefore increasing.

\subsection*{The critical case $\rho_2{=}C$}
We easily get the relation
\begin{equation}\label{eqPhi3}
  \lim_{z\to+\infty} \frac{z^2\Phi_{\rho_1,C}(z)}{(\rho_1{+}C z)^C}C!=\frac{\rho_1{-}C{+}1}{C}.
\end{equation}
If $\rho_1{<}C{-}1$, then  $\Phi_{\rho_1,C}$ has negative values and, by Relation~\eqref{eqPhi1},  a root $z_C{>}0$. It is unique, otherwise it would imply that $f_{\rho_1,C}$ has two distinct extrema which is not possible by Relation~\eqref{eqPhi2}. The limit $z(\rho_1,C)$ is necessarily a root of $\Phi_{\rho_1,C}$, hence, $z(\rho_1,C){=}z_C$.

If $\rho_1{\ge}C{-}1$. Relation~\eqref{eqPhi2} gives  that $f_{\rho_1,C}$ is strictly decreasing. Since it is converging to $0$,   we conclude that $\Phi_{\rho_1,C}$ does not have a root in this case and thus, necessarily $z(\rho_1,C){=}{+}\infty$.

\subsection*{The overloaded case $\mathbf{\rho_2}{>}C$}
By Relation~\eqref{eqPhi1}, the function $\Phi_{\rho_1,\rho_2}$ converges to ${+}\infty$ at $0$ and at ${+}\infty$. To determine the number of roots of  $\Phi_{\rho_1,\rho_2}$,  one has therefore to obtain the sign of $m(\rho_1,\rho_2)$ defined by Equation~\eqref{phim}.

Clearly $\rho_1{\to}m(\rho_1,\rho_2)$ is a continuous function on $\R_+$. For $\rho_1^0{<}\rho_1^1$ and $z{>}\rho_1^1/\rho_2$,  the relation
 \[
\Phi_{\rho_1^0,\rho_2}(z{-}\rho_1^0/\rho_2) <\Phi_{\rho_1^1,\rho_2}(z{-}\rho_1^1/\rho_2)
 \]
gives $m_{\rho_1^0}{\le}m_{\rho_1^1}$, $\rho_1{\to}m(\rho_1,\rho_2)$ is also an increasing function.   Note that, for  $z{>}0$,
 \[
 \lim_{\rho_1\to 0} \Phi_{\rho_1,\rho_2}(\rho_1 z)=-1,
 \]
 hence if $\rho_1$ is sufficiently small then $m(\rho_1,\rho_2){<}0$. If $\rho_1{>}C{-}1$, then, from Relation~\eqref{eqPhi2} we get that the function $f_{\rho_1,\rho_2}$ is decreasing, and converging to $0$ at infinity. Consequently $\Phi_{\rho_1,\rho_2}$ is positive on $\R_+$ and therefore $m(\rho_1,\rho_2){>}0$. Hence,  $ \phi_C$ defined by Equation~\eqref{phim}
satisfies  $\phi_C(\rho_2){\in}(0,C{-}1)$. By continuity of $\rho_1{\to}m(\rho_1,\rho_2)$, we have $m_{\phi_C(\rho_2)}{=}0$, the function $\Phi_{\phi_C(\rho_2),\rho_2}$ has therefore a unique root. 

 By definition,  $\rho_1{>}\phi_C(\rho_2)$ if and only if $m(\rho_1,\rho_2){<}0$ or, equivalently, since the minimum of $\Phi_{\rho_1,\rho_2}$ is necessarily achieved in some compact interval of $(0,{+}\infty)$,  if the relation
 \[
 \sup_{z{>}0}\left(z\sum_{m=0}^{C-1}\frac{C!}{m!} \frac{1}{(\rho_1{+}\rho_2 z)^{C-m}}\right)<1
 \]
 holds. We deduce that  $\phi_C$ is a non-increasing function on $(C,{+}\infty)$ and that $\phi_C(\rho_2)$ converges to $0$ as $\rho_2$ gets large. 
 If
 \[
 \delta_C\steq{def}\lim_{\mathclap{\rho_2\searrow C}}\uparrow\phi_C(\rho_2)< C{-}1,
 \]
one can take $\rho_1{\in}(\delta_C,C{-}1 )$. From the critical case, we know that $\Phi_{\rho_1,C}$ has exactly one root and that there exists $z_1{>}0$ such that $\Phi_{\rho_1,C}(z_1){<}0$. One can fix  $\rho_2^1{>}C$  sufficiently close to $C$ so that $\Phi_{\rho_1,\rho_2}(z_1){<}0$. Hence, $\Phi_{\rho_1,\rho_2^1}$ has a root and, consequently, $\rho_1{\le}\phi_C(\rho_2^1)$ which contradicts the fact that $\rho_1{>}\delta_1{\ge}\phi_C(\rho_2^1)$. The proposition is proved. 
\end{proof}

\noindent
{\bf Examples.}
\begin{enumerate}
\item When $C{=}2$ and $\rho_1{<}2{<}\rho_2$, one has
  \[
  2\Phi_{\rho_1,\rho_2}(z)= \rho_2\left( \rho_2{-}2 \right) {z}^{2} {+} 2\left( \rho_1\rho_2{-}\rho_1{-}1\right) z{+} {\rho_1}^{2}.
  \]
 The function $\Phi(\rho_1,\rho_2)$ has two roots on $\R_+$ when $\Phi_{\rho_1,\rho_2}'(0){<}0$, that is, if $\rho_1{<}1/(\rho_2{-}1)$, and
if the minimum of $\Phi_{\rho_1,\rho_2}$  is negative, that is
  \[
  \rho_1^2+2(1{-}\rho_2)\rho_1+1>0.
  \]
  It is then easy to deduce that
  \[
  \phi_2(\rho_2)=\rho_2{-}1{-}\sqrt{\rho_2(\rho_2{-}2)}.
  \]
\item When $C{=}3$ and $\rho_1{<}3{<}\rho_2$, one has
  \[
  6\Phi_{\rho_1,\rho_2}(z){=}{\rho_2}^{2} \left(\rho_2{-}3 \right) {z}^{3}{+}3\rho_2 \left( \rho_1\rho_2{-}2\rho_1{-}2 \right) {z}^{2}
{+}3 \left( {\rho_1}^{2}\rho_2{-}{\rho_1}^{2}{-}2\rho_1{-}2 \right) z{+}{\rho_1}^{3}.
\]
The discriminant of this polynomial (in $z$) is
\[
H_{\rho_2}(u){\steq{def}}3{\rho_1}^{4}{+} 2\left(6{-}5\rho_2\right) {u}^{3}{+}3\left(3{\rho_2}^{2}{-}8(\rho_2{-}1)\right){u}^{2}{-}12\left(\rho_2{-}2\right) u
{-}8\rho_2{+}12.
\]
Since $H_{\rho_2}(0){=}{-}12(\rho_2{-}2){<}0$  holds and $H_{\rho_2}(C{-}1){=}4(9\rho_2{-}25))(\rho_2{-}3){>}0$   for $\rho_2{>}3$,  it is then easily seen that 
\[
\phi_3(\rho_2)=\inf\{u:H_{\rho_2}(u){=}0\}.
\]
\end{enumerate}
\medskip
\subsection{RIST Algorithm with One Retrial}\label{Ript1sec}
Proposition~\ref{PhiProp} shows that the dynamical system associated to the RIST with an infinite number of retrials has at most two equilibrium points. As noted earlier, there is another equilibrium which is not described here, it is investigated in Section~\ref{sec1-Sat}.

The purpose of this section is in showing that in the case when only one attempt to accommodate a request is allowed, there are also cases with three equilibrium points but all of them are equilibria of the ``smooth'' dynamical system. 

In view of Proposition~\ref{DynSysProp2} and in the same as the derivation of Relation~\eqref{Phi},  an equilibrium point is of the form $\pi_S$, with 
\begin{equation}\label{piS}
\pi_S(x,y)=\frac{1}{Z_S} \cdot \frac{\rho_1^x}{x!}\frac{(\rho_2 S)^y}{y!},\quad (x,y){\in}{\cal X},
\end{equation}
where $S$ is a solution $z{\in}(0,1)$ of
\begin{equation}\label{FixpiS}
\frac{(\rho_1{+}\rho_2 z)^C}{C!}-z\sum_{m=0}^{C} \frac{ (\rho_1{+}\rho_2 z)^m}{m!}=0,
\end{equation}
and $Z_S$ is the normalization constant.

With the change of coordinates $z{\mapsto}(z{-}\rho_1)/\rho_2$, this amounts to finding the roots $z{\in}(\rho_1,\rho_1{+}\rho_2)$ of $g$, with 
\[
g(z)\steq{def}\frac{\rho_2}{z{-}\rho_1}\frac{z^C}{C!}e^{-z}-\sum_{m=0}^{C} \frac{ z^m}{m!}e^{-z}.
\]
Note that $g(s){\to}{+}\infty$ when $s{\searrow}\rho_1$ and $g(\rho_1{+}\rho_2){<}0$. Simple calculations give the relation
\[
g'(z)=\frac{z^{C-1}}{C!}e^{-z}f(z),
\]
with $f(z){=}z^3{-}(2\rho_1{+}\rho_2) z^2{+} (\rho_1^2{+}\rho_1\rho_2{+}\rho_2(C{-}1)) z{-}C\rho_1\rho_2$. It shows in particular that Equation~\eqref{FixpiS} cannot have more than three solutions.

We give a scaling picture of the fixed point equation~\eqref{FixpiS}. As for the DAR algorithm which is investigated in the next section, we study the case when the capacity $C$ is a scaling parameter going to infinity and $\rho_2$ is of the order of $C$, i.e.\ $\rho_2{=}\nu_2 C$, for some $\nu_2{>}0$.

Under some condition, there is always a solution of Equation~\eqref{FixpiS} close to $0$. For that we do the change of variable $z{\mapsto}\rho_1z$, the relation becomes 
\[
\psi_{1,C}(z)\steq{def} -z+\rho_1^{C{-}1}\frac{({1+}\nu_2C z)^C}{C!}-z\sum_{m=1}^{C} \rho_1^m\frac{ (1{+}\nu_2C z)^m}{m!}=0.
\]
With Stirling's Formula, it easily seen that, for $\eps{>}0$,
\[
\limsup_{C\to{+}\infty} \psi_{1,C}(\eps)\leq -\eps+\lim_{C\to{+}\infty}\rho_1^{C{-}1}\frac{({1+}\nu_2C \eps)^C}{C!}=-\eps,
\]
provided that $\nu_2\rho_1{\geq}e$. Since $\psi_{1,C}(0){>}0$, we get that, for $C$ sufficiently large, Equation~\eqref{FixpiS} has a solution in the interval $(0,(\rho_1\eps){\wedge}1)$.

Returning to Equation~\eqref{FixpiS}, it can be written as 
\[
\psi_{2,C}(z)\steq{def}z\sum_{k=0}^{C}\frac{1}{(\rho_1{+}\nu_2C z)^k} \frac{C!}{(C{-}k)!}-1=0.
\]
Assuming that $\nu_2 z{>}1$, we can check that
\[
\lim_{C\to{+}\infty} \psi_{2,C}(z)=\Psi_{2,\infty}(z)\steq{def}\frac{\nu_2z^2{-}\nu_2z{+}1}{\nu_2z{-}1},
\]
and the convergence is uniform on any compact interval of $(1/\nu_2,1]$. If $\nu_2{>}4$, the function $\Psi_{2,\infty}$ has two zeroes in the interval $(1/\nu_2,1)$ given by
\[
z^*_{\pm}{\steq{def}}\frac{1\pm\sqrt{1{-}4/\nu_2}}{2}.
\]
We now summarize this result in the following proposition.
\begin{proposition}\label{prop1T}
  Under the assumption that $\rho_2{=}\nu_2C$ and if $\nu_2{>}\max(4,e/\rho_1)$, there exists $C_0{>}0$ such that if $C{\ge}C_0$,  the dynamical system~\eqref{DynSys2} associated to the RIST algorithm with one retrial  has exactly three equilibrium points converging respectively to
\[
0,\; \frac{1{-}\sqrt{1{-}4/\nu_2}}{2}, \; \frac{1{+}\sqrt{1{-}4/\nu_2}}{2},
\]
as $C$ goes to infinity.
\end{proposition}
The proof is skipped since most of the arguments have been given and the proof of Theorem~\ref{ThNbFP} in the next section is similar and slightly more technical. 
\subsection{Stability of Saturation}\label{sec1-Sat}
\addcontentsline{toc}{section}{\thesubsection \hspace{3.5mm} Stability of Saturation}
In this section it is assumed that $\rho_2{>}C$. The main result is that if the initial state is sufficiently congested, so is the state of the network  on any finite time interval.

We describe the general strategy of our approach. Note first that the two-dimensional process of the total number of empty places and total number of class~$1$ customers  in the network does not have the Markov property. Indeed in state $(m,n)$ with $m{>}0$, when an external job arrives it is not possible to determine if the transition to $(m{-}1,n)$ or to $(m{-}1,n{+}1)$ occurs, i.e. if the job is blocked at its arrival queue or not.  This process  can in fact be compared with  an ergodic Markov process in $\N^2$ by using a  convenient coupling and a  specific order relation in $\N^2$.  The ergodicity property is then used to show that, asymptotically, the  total number of empty places and of class~$1$ customers in the network is negligible with respect to $N$, so that the entire system is composed of rerouted jobs in the limit. 

\medskip

If $(Z^N(t)){=}(X_i^N(t),Y_i^N(t))$ is the solution of the SDE~\eqref{SDEX} and~\eqref{SDEY}, define
\[
\overline{Z}_1^N(t){\steq{def}}\sum_{i=1}^N X_i^N(t),\quad  \overline{Z}_2^N(t){\steq{def}}\sum_{i=1}^N Y_i^N(t), 
\]
and $\overline{Z}_0^N(t){=}CN{-}\overline{Z}_1^N(t){-}\overline{Z}_2^N(t)$ is the total number of empty places at time $t$.

\medskip

\noindent
{\sc A Coupling with a Two-Dimensional Markov Process.}
We introduce a Markov process in $\N^2$ which will be used in the analysis of the asymptotic behavior of the  process of the empirical distribution associated to $(Z^N(t))$.
\begin{definition}\label{Qmat}
Let $(U^N(t){\steq{def}}(U_0^N(t),U_1^N(t))$ be the Markov process on $\N^2$ with the initial state $(\overline{Z}_0^N(0),\overline{Z}_1^N(0))$, and $Q$-matrix $Q{=}(q({\cdot},{\cdot}))$ defined on the state space ${\cal S}_U{=}\{u{=}(u_0,u_1){\in}\N^2:u_0{+}u_1{\leq} N\}$,
\[
\begin{cases}
q(u,u{-}e_0{+}e_1){=}\lambda u_0, \\
q(u,u{-}e_0){=}\lambda(N{-}u_0),\\
q(u,u{+}e_0{-}e_1){=}\mu_1u_1,\\
q(u,u{+}e_0){=}\mu_2(CN{-}u_0{-}u_1),
\end{cases}
\]
with $e_0{=}(1,0)$ and $e_1{=}(0,1)$, provided that the transitions keep the process in ${\cal S}_U$, and define
\[
U_2^N(t){\steq{def}}CN{-}U_0^N(t){-}U_1^N(t).
\]
\end{definition}
With some abuse of notation, we will also speak of the $Z$ and $U$-systems to refer to the associated stochastic processes  $(Z^N(t))$ and $(U^N(t))$ and, similarly of class~$1$ and~$2$ jobs in the $U$-system with an obvious meaning. Finally, 
\begin{equation}\label{SU}
S_U^N\steq{def}\inf\left\{t{\geq} 0: U_0^N(t){+}U_1^N(t){=}N\right\}
\end{equation}
and, for $a{\in}\N$,
\begin{equation}\label{TU}
T_{U}^a\steq{def}\inf\left\{t{\geq} 0: U_0^N(t){\geq}a\right\}.
\end{equation}

The transition rates of the process $(U^N(t))$ suggest that this process would behave as  $(Z^N(t))$ if all nodes had, at most, one empty place. The coupling shows that the process  $(Z^N(t))$ can be upper bounded in some way by such process. A key element in the coupling is the explicit use of the fact that services times of class~1 jobs are ``smaller'' than the services times of class~2 jobs.

\begin{proposition}\label{propCoup}
There exists a coupling of the processes  $(U_0^N(t),U_1^N(t))$ and $(Z^N(t))$ such that  $(U_0^N(0),U_1^N(0)){=}(\overline{Z}_0^N(0),\overline{Z}_1^N(0))$ and that the relations
\begin{equation}\label{eqc}
  \begin{cases}
U_2^N(t)\leq  \overline{Z}_2^N(t)\\
U_1^N(t){+}U_2^N(t)\leq   \overline{Z}_1^N(t){+}\overline{Z}_2^N(t)
  \end{cases}
  \end{equation}
hold for all $t{<}S_U^N$,
where $S_U^N$ is defined by Relation~\eqref{SU}.
\end{proposition}
\begin{proof}
  We proceed by induction on the number of jumps. One has  to show that if the relation holds initially then it will also hold at the first jump of    $(Z^N(t))$ or  $(U_0^N(t),U_1^N(t))$.

For $j{\in}\{0,1\}$, define  $z_j{\steq{def}}\overline{Z}_j^N(0)$ and  $u_j{\steq{def}}U_j^N(0)$. By assumption,
\[
\begin{cases}
u_2\le z_2\\
u_1{+}u_2\le z_1{+}z_2,
\end{cases}
\]
we can assume that $u_0{<}N$ since the process is stopped at time $S_U^N$. We define
\[
\widetilde{a}_0=\sum_{i=1}^N \ind{X_{i}^N(0)+Y_i^N(0)<C},
\]
the number of non-saturated queues.  Clearly,
\[
\widetilde{a}_0\le z_0{=}CN{-}z_1{-}z_2\le CN{-}u_1{-}u_2=u_0.
\]
We will take the convention that $E_\xi$ denotes an exponentially distributed random variable with parameter $\xi{\ge}0$ and that all exponential random variables constructed are independent. The coupling is done by introducing the following random variables; the minimum of them will define the first jump of the process. For each random variable, the transition is indicated for $(\overline{Z}^N(t)){=}(\overline{Z}_0^N(t),\overline{Z}_1^N(t))$ and $(U^N(t))$  in the case it has the minimal value. 

\medskip
\begin{enumerate}
\item  Arrivals.
    \begin{enumerate}
    \item  $E_{\lambda \widetilde{a}_0}$ is the minimum of the arrivals of jobs finding a non-congested queue in the $Z$-system;\\
      {\tt Transition:} $z{\mapsto}z{+}e_1{-}e_0$ and $u{\mapsto} u{+}e_1{-}e_0$.
\item  $E_{\lambda ({u}_0-\widetilde{a}_0))}$ is the minimum of the remaining arrivals of jobs finding a non-congested queue in the $U$-system;\\
      {\tt Transition:} $u{\mapsto} u{+}e_1{-}e_0$ and,  if $\widetilde{a}_0{>}0$, $z{\mapsto}z{-}e_0$.
\item  If $u_0{>}0$. $E_{\lambda (N-u_0)}$ is  the minimum of the arrivals of jobs finding a congested queue in the $Z$-system and the $U$-system.\\
      {\tt Transition:} $u{\mapsto}u{-}e_0$ and, if $\widetilde{a}_0{>}0$, $z{\mapsto} z{-}e_0$.
    \end{enumerate}
\item Services

  \begin{enumerate}
    \item   $E_{\mu_2 u_2}$ is the minimum of the services of  $u_2$ class~$2$ jobs of the $U$-system and $Z$-system;\\
      {\tt Transition:} $z{\mapsto}z{+}e_0$ and $u{\mapsto} u{+}e_0$. Recall that $u_2{\le}z_2$.
\item If $u_1{\le}z_1$.
    \begin{enumerate}
      \item  $E_{\mu_1 u_1}$ is the minimum of the services of  $u_1$ class~$1$ jobs of the $U$-system and $Z$-system;\\
      {\tt Transition:} $z{\mapsto}z{+}e_0{-}e_1$ and $u{\mapsto} u{+}e_0{-}e_1$.

      \item $E_{\mu_1 (z_1-u_1)}$ is the minimum of the remaining services of the $Z$-system. In this case, note that one has necessarily $z_1{>}u_1$;\\
      {\tt Transition:} $z{\mapsto}z{+}e_0{-}e_1$ and $u{\mapsto} u$.

      \item   $E_{\mu_2 (z_2-u_2)}$ is the minimum of the services of  the remaining $z_2-u_2$ class~$2$ jobs of $Z$-system. In this case $z_2{>}u_2$;\\
      {\tt Transition:} $z{\mapsto}z{+}e_0$ and $u{\mapsto} u$.

    \end{enumerate}
\item If $u_1{>}z_1$. We fix $F{=}E_{u_1{-}z_1}$.
  \begin{enumerate}
\item   $E_{\mu_1z_1}$ is the minimum of the services of  $z_1$ class~$1$ jobs of the $U$-system and $Z$-system;\\
      {\tt Transition:} $z{\mapsto}z{+}e_0{-}e_1$ and $u{\mapsto} u{+}e_0{-}e_1$.

\item   $F/\mu_1$ is the minimum of the services of  $u_1{-}z_1$ remaining class~$1$ jobs of the $U$-system;\\
      {\tt Transition:} $z{\mapsto}z$ and $u{\mapsto} u{+}e_0{-}e_1$.

\item $F/\mu_2$ is the minimum of services of some $u_1{-}z_1$  class~$2$ jobs of the $Z$-system\\
      {\tt Transition:} cannot be the next step  since $F/\mu_1{<}F/\mu_2$, the transition of (2)~(c)~(ii) occurs therefore before.

\item $E_{\mu_2(z_2{-}u_2{-}(u_1{-}z_1)}$ is the minimum of services of the remaining $\mu_2(z_2{-}u_2{-}(u_1{-}z_1)$  class~$2$ jobs of the $Z$-system.  In this case,  $u_1{+}u_2{<}z_1{+}z_2$;\\
      {\tt Transition:} $z{\mapsto}z{+}e_0$ and $u{\mapsto} u$.
\end{enumerate}
\end{enumerate}
\end{enumerate}
An easy, but somewhat tedious, check  gives that the two processes have the correct time evolution and, furthermore that the order relation is preserved after any of the transitions mentioned above. 
\end{proof}

\medskip

\noindent
{\sc An Asymptotic Analysis of $(U^N(t))$.}
It is assumed  that the initial state of the process $(U^N(t))$ of Definition~\eqref{Qmat} satisfies the relation
\begin{equation}\label{InitU}
\lim_{N\to+\infty} \frac{1}{N}\left(U_0^N(0),U_1^N(0)\right)=(a_0,a_1).
\end{equation}
\begin{proposition}\label{AsympU}
Under the  condition  $\rho_1{<}C{<}\rho_2$, there exists $\eta_0{>}0$  such that, if the initial conditions~\eqref{InitU} satisfies the relations $0{\le}a_0{+}a_1{\le}\eta_0$ then for any $\eps{>}0$, there is  $t_0{>}0$ and  a constant $K_0$ such that,  for any $T{>}0$,
\[
\lim_{N\to+\infty}\P\left(\sup_{t_0{\le} t{\le} t_0{+}T}U_0^N(t){\leq} K_0\log N, \sup_{t_0{\le} t{\le} t_0{+}T}\frac{U_1^N(t)}{N}\leq \eps,\quad S_U^N{\geq}t_0{+}T\right)=1.
\]
\end{proposition}
\begin{proof}
  For $0{<}a_0{+}a_1{<}\eta{<}1$, by Definition~\ref{Qmat} of the process $(U^N(t))$,   a simple coupling shows that the process $(U_1^N(t{\wedge}S_U^N{\wedge}T_U^{\eta N}))$ can be stochastically  upper bounded by $(L_\eta^N(t{\wedge}S_U^N{\wedge}T_U^{\eta N}))$, where $(L_\eta^N(t))$  is the process of the number of jobs of an $M/M/\infty$ queue with arrival rate $\lambda\eta N$ and service rate $\mu_1$ and initial point $U_1^N(0)$, and $T_U^{\eta N}$ is defined by Relation~\eqref{TU}. A classical result, see Theorem~6.13 of Robert~\cite{Robert} for example,  gives the following convergence in distribution
\begin{equation}\label{eq1-2}
\lim_{N\to+\infty} \left(\frac{L_\eta^N(t)}{N}\right)=\left(\frac{\lambda}{\mu_1}\eta{+}\left(a_1{-}\frac{\lambda}{\mu_1}\eta\right)e^{-\mu_1t}\right).
\end{equation}
With a similar argument, the process $(U_0^N(t{\wedge}S_U^N{\wedge}T_U^{\eta N}))$ can be upper bounded by $(Q((Nt){\wedge}S_U^N{\wedge}T_U^{\eta N}))$, where $(Q(t))$ is the process of the number of jobs of  an $M/M/1$ queue with respective arrival and service rates $\mu_2C{+}\eta(\mu_1{-}\mu_2)$ and $\lambda$, and  initial point $U_0^N(0)$. Again a classical result, see Proposition~5.16 of Robert~\cite{Robert} for example,  gives the following convergence in distribution
\begin{equation}\label{eq3}
\lim_{N\to+\infty}\left(\frac{Q(Nt)}{N}\right)=\left(a_0+\left((\mu_1{-}\mu_2)\eta{+}\mu_2C-\lambda\right)t\right)^+,
\end{equation}
where $a^+{=}\max(a,0)$ for $a{\in}\R$.

Now, we fix $0{<}\eta_0{<}1$ such that 
\begin{equation}\label{eq5}
(\mu_1{-}\mu_2)\eta_0{+}\mu_2C{<}\lambda \text{ {\rm and } } \eta_0\frac{\lambda}{\mu_1}{<}1.
\end{equation}
If $(a_0,a_1)$ is such that
\[
a_0{<}\eta_0{\wedge}(1{-}\lambda\eta_0/\mu) \text{ {\rm and } } a_1{<}{\lambda}\eta_0/{\mu_1},
\]
then, by using Relations~\eqref{eq1-2} and~\eqref{eq3}, we get that
\[
\lim_{N\to+\infty}  \P\left(\sup_{0\le s\leq t}\frac{L_{\eta_0}^N(s)}{N}<\eta_0, \sup_{0\le s\leq t}\frac{1}{N}(L_{\eta_0}^N(s)+Q(Ns))< 1\right)=1.
\]
This implies,  in particular,  the relation
\begin{equation}\label{eq4}
\lim_{N\to+\infty} \P\left(\min\left(S_U^N,T_U^{\eta_0N}\right){>}t\right)=1
\end{equation}
for all $t{\ge}0$. Therefore, Relations~\eqref{eq1-2} and~\eqref{eq3} hold with $(L_{\eta_0}^N(t))$ [resp. $(Q(Nt))$] replaced by $(U_1^N(t))$ [resp. $(U_0^N(t))$].
Additionally, Relation~\eqref{eq3} shows that, 
\[
\lim_{N\to+\infty}  \P\left(T_U^0\leq \frac{\eta_0}{\lambda{-}(\mu_1{-}\mu_2)\eta_0{+}\mu_2C}\right)=1.
\]

With the same coupling as before and the strong Markov property, the process $(U_0^N(T_U^0{+}t), 0{\le}t{\le}T)$ is upper bounded by $(Q(Nt),0{\le}t{\le}T)$, where $(Q(t))$ is the same $M/M/1$ process as before but starting at $0$, $Q(0){=}0$. Let
\[
H_b\steq{def}\inf\{t{\ge}0: Q(t)=b\},
\]
Proposition~5.11 of Robert~\cite{Robert} shows that, if $\rho{\steq{def}}(\mu_1{-}\mu_2)\eta_0{+}\mu_2C)/\lambda{<}1$, then, as $b$ goes to infinity, the sequence of random variables $(\rho^bH_b)$ converges in distribution to an exponential distribution.  Define
\[
{\cal A}_N{\steq{def}}\left\{\sup_{T_U^0\le s\leq T_U^0+T} U_0^N(s) \leq C\log N\right\},
\]
by choosing $C{>}{-}\log\rho$, one has therefore the relation
\[
\liminf_{N\to+\infty}\P\left({\cal A}_N\right)\geq \lim_{N\to+\infty}\P\left(\sup_{0\le s\leq NT} Q(s) \leq C\log N\right)=1. 
\]
The proposition is proved.
\end{proof}
Propositions~\ref{propCoup} and~\ref{AsympU} give the following proposition. 
\begin{proposition}[Stability of Saturation]\label{theo1}
Under the  condition  $\rho_1{<} C{<}\rho_2$, there exists some $\eta_0{>}0$ and $t_0{>}0$ such that, if the initial condition is such that 
  \begin{equation}\label{Init}
\liminf_{N\to{+}\infty} \frac{1}{N}\sum_{i=1}^N Y_{i}^N(0) \geq C{-}\eta_0,
  \end{equation}
  then, for any $\eps{>}0$ and $T{\ge}0$,
  \[
  \lim_{N\to+\infty} \P\left(\inf_{t_0\leq s\leq t_0{+}T} \frac{1}{N}\sum_{i=1}^N Y_i^N(s) \geq C{-}\eps\right)=1.
  \]
\end{proposition}
\begin{proof}
With the above notations,
\[
\overline{Z}_1^N(t){\steq{def}}\sum_{i=1}^N X_i^N(t),\quad  \overline{Z}_2^N(t){\steq{def}}\sum_{i=1}^N Y_i^N(t), 
\]
and $\overline{Z}_0^N(t){=}CN{-}\overline{Z}_1^N(t){-}\overline{Z}_2^N(t)$. Without loss of generality, by taking a subsequence for example, we can assume that
\[
\lim_{N\to{+}\infty} \frac{\overline{Z}_1^N(0)}{N}=a_1\text{  and  }\lim_{N\to{+}\infty} \frac{\overline{Z}_0^N(0)}{N}=a_0,
\]
for $a_0$, $a_1{\in}[0,C]$. Proposition~\ref{propCoup} gives a coupling of  the process $(\overline{Z}_0^N(t),\overline{Z}_1^N(t))$ with the process $(U_0^N(t),U_1^N(t))$, with the same initial conditions, such that the relation $U_2^N(t){\le}\overline{Z}_2^N(t)$ holds for $t{<}S_U^N$.

Proposition~\ref{AsympU} shows that there exists $\eta_0{>}0$ and $t_0{\ge}0$ such that if Relation~\eqref{Init} holds then $a_0{+}a_1{\le}\eta_0$, so that, for $T{>}0$,
\[
\lim_{N\to+\infty}\P\left(\sup_{t_0{\le} t{\le} t_0{+}T}\frac{U_0^N(t)}{N}{\leq} \eps, \sup_{t_0{\le} t{\le} t_0{+}T}\frac{U_1^N(t)}{N}\leq \eps,\, S_U^N{\geq}t_0{+}T\right)=1.
\]
With the coupling, we obtain the relation
\[
\lim_{N\to+\infty}\P\left(\inf_{t_0{\le} t{\le} t_0{+}T}\frac{\overline{Z}_2^N(t)}{N}{\geq} C{-}2\eps\right)=1.
\]
The proposition is proved.
  \end{proof}
The next result shows a stability property of the saturated state $(0,C)$ of the mean-field limit of the process $(X^N(t))$,
\begin{corollary}
Under the assumptions of Proposition~\ref{theo1}, any limiting point $(\Lambda(t))$ of the sequence of empirical processes $(\Lambda^N(t))$ of Relation~\eqref{Emp} satisfies the following relation for the convergence in distribution, 
  \[
  \Lambda(t)((0,C))>1{-}\eps, \,\forall t{\ge}t_0.
    \]
\end{corollary}
\begin{proof}
  This is simply due to the above proposition and the fact that

  \[
  \frac{1}{N}\sum_{i=1}^N \ind{Y_i^N(s){\ne}C} \leq   \frac{1}{N}\sum_{i=1}^N (C{-}Y_i^N(s))= C-\frac{\overline{Z}_2^N(t)}{N}
  \]
\end{proof}
\subsection{A Spectral Criterion of Stability}\label{sec1-Spec}
\addcontentsline{toc}{section}{\thesubsection \hspace{3.5mm} A Spectral Criterion of Stability}

In this section we investigate the stability properties of the non-linear dynamical system $(\Lambda(t))$ defined by Relation~\eqref{DynSys}. The corresponding linear system $(\Lambda_R(t))$ is the solution of Relation~\eqref{LinDynSys}. The goal of this section is of showing that if the linear process $(\Lambda_R(t))$ is converging sufficiently fast to equilibrium, the non-linear process will converge to this invariant distribution provided its initial state is  sufficiently close to it. 

As before,  $\pi_R$ is the probability distribution on ${\cal X}$ defined by Relation~\eqref{piR}, it is the invariant measure of $(\Lambda_R(t))$. If $R$ is a solution of Equation~\eqref{PiRFix} which has been studied in Section~\ref{1LocSec}, $\pi{=}\pi_R$ is an invariant measure of $(\Lambda(t))$. Following Aldous and Fill~\cite{AldousFill}, the ``distance'' between $\mu{\in}{\cal P}({\cal X})$ and a fixed probability $\pi$ is defined as 
\[
\|\mu{-}\pi)\|_2^2\steq{def} \sum_{z\in{\cal X}} \left(\frac{\mu(z)}{\pi(z)}{-}1\right)^2 \pi(z)
=\sum_{z\in{\cal X}}\frac{(\mu(z)-\pi(z))^2}{\pi(z)}.
\]
Lemma~3.26 of Aldous and Fill~\cite{AldousFill} shows that there exists a maximal $\kappa_R{>}0$ such that, for all $t{\ge}0$, 
\begin{equation}\label{e121}
\frac{\diff}{\diff t} \|\Lambda_R(t){-}\pi\|_2^2\leq {-}2\kappa_R\|\Lambda_R(t){-}\pi\|_2^2.
\end{equation}
This is the classical exponential convergence to equilibrium for finite Markov processes, the distance $\|\cdot\|_2$ gives the nice Inequality~\eqref{e121} of such phenomenon. The quantity $\kappa_R$ is the {\em spectral gap} of the process $(\Lambda_R(t))$, see Theorem~3.25 of Aldous and Fill~\cite{AldousFill} for a variational characterization.
\begin{theorem}\label{t138}
Let $R$ be a solution of Equation~\eqref{PiRFix} and assume that the spectral gap $\kappa_R$ of $(\Lambda_R(t))$ satisfies the condition
\begin{equation}\label{e128}
\kappa_R>\frac{\lambda}{1{-}R}\sqrt{\frac{C}{\rho_2}},
\end{equation}
then there exist positive constants $q$ and $\eps_0$ such that if, $\|\Lambda(0){-}\pi\|_2\le\eps_0$, then the relation 
\begin{equation}\label{e115}
\frac{\diff}{\diff t}\|\Lambda(t){-}\pi\|_2^2\le{-}q \|\Lambda(t){-}\pi\|_2^2, \quad \text{ for all } t\ge 0
\end{equation}
holds with $\pi{=}\pi_R$ defined by Relation~\eqref{piR}. In particular, $\pi$ is an exponentially stable equilibrium point of $(\Lambda(t))$.
\end{theorem}

\bigskip

\noindent
{\sc Strategy of the proof}.
We first describe the main ideas. Our technique is based on a variation of an argument coming from the theory of attractors. The original approach consists of splitting the half-line $\R_+$ into intervals of large length and on each of them  the dynamical system is separated into two parts.  The first one takes the linear part of the equation for which there is an exponential convergence. The second one includes the non-linearity and is issued from zero on each interval. The convergence of the first part is then used to absorb the second part of the flow originated after the splitting and allows to obtain the desired result. We refer the reader to the paper Zelik~\cite{Zelik} where this powerful approach is used in another context. Unfortunately, the direct application of this technique does not lead to good results in our case  due to the strong non-linearity in Relation~\eqref{DynSys}. To overcome this difficulty, we shall first prove an  ``instantaneous" absorbing and then use a bootstrap argument to propagate; see Step 2 of the proof below.

\begin{proof}

\medskip
\noindent
{\sc Step 1: Splitting of the flow}. Let $\Lambda_R(t)$ be the solution of Equation~\eqref{LinDynSys} with a fixed initial point $\Lambda(0)$ and
  \[
  F(t){\steq{def}}\|\Lambda(t){-}\pi\|_2^2.
  \]
For $t{\ge}0$, we denote by $f'(t)$ the derivative of a differentiable function $f$ at $t$, we have
\[
  {F}'(t)=\frac{\diff}{\diff t}\sum_{z\in{\cal X}} \frac{([\Lambda(t){-}\pi](z))^2}{\pi{(z)}}=2\sum_{z{\in}{\cal X}} \frac{([\Lambda(t){-}\pi](z){\Lambda}'(t)(z)}{\pi{(z)}}.
  \]

We will now state two claims and show how they are used to establish our theorem. The proof of the claims conclude the proof of the theorem.
 \item {\em Claim~1}. The relation
\begin{equation}\label{e120} 
{\Lambda}'(0)(z)={\Lambda}_R'(0)(z)+\lambda L_{\Lambda, R}\left(\rule{0mm}{4mm}\Lambda(0)(z{-}e_2){-}\ind{z\in{\cal X}_+}\Lambda(0)(z)\right),
\end{equation}
holds, where 
\begin{equation}\label{LR}
L_{\Lambda, R} = \frac{\Lambda(0)({\cal X}_+^c)}{1{-}\Lambda(0)({\cal X}_+^c)}-\frac{R}{1{-}R}.
\end{equation}

 Taking $t{=}0$ and using that $\Lambda_R(0){=}\Lambda(0)$ together with Relation~\eqref{e120}, we obtain
\begin{multline}\label{e125}
{F}'(0)=2\sum_{z\in{\cal X}} \frac{([\Lambda_R(0){-}\pi](z)){\Lambda}_R'(0)(z)}{\pi(z)}\\
+2\lambda L_{\Lambda, R}\sum_{z\in{\cal X}} \frac{([\Lambda(0)-\pi](z))(\Lambda(0)(z{-}e_2){-}\ind{z\in{\cal X}_+}\Lambda(0)(z))}{\pi(z)}={\cal I}_1{+}{\cal I}_2.
\end{multline}
Note that the term  ${\cal I}_1$ is the derivative of the function  $t{\mapsto}\|\Lambda_R(t){-}\pi\|_2^2$ at zero. From Inequality~\eqref{e121}, we get therefore that 
\begin{equation}\label{AlFi}
\frac{{\cal I}_1}{2}\le {-}\kappa_R\|\Lambda_R(0){-}\pi\|_2^2={-}\kappa_R\|\Lambda(0)-\pi\|_2^2=-\kappa_R F(0).
\end{equation}

\smallskip
    \item  {\em Claim~2}. There exist $\eps{>}0$ and $\theta{<}2\kappa_R$ such that, for $\|\Lambda(0){-}\pi\|_2{\le} \eps$, 
\begin{equation}\label{e401}
 {\cal I}_2\le \theta F(0).
\end{equation}
 Combining this with ~\eqref{e125} and \eqref{AlFi}, we get
\begin{equation}\label{e402}
{F}'(0)\le -q F(0)
\end{equation}
for $\|\Lambda(0){-}\pi\|_2{\le} \eps$ and $q{=}2\kappa_R{-}\theta{>}0$.

\medskip

\noindent
    {\sc Step 2: Bootstrap argument.}
First note that we may suppose that $F(0){>}0$ holds since,  otherwise $\Lambda(0){=}\pi$ and therefore $\Lambda(t){=}\pi$ for all $t{\ge}0$. There is nothing to prove in this case.  Let
\[
\tau\steq{def}\inf\left\{s{>}0: F(s){>}\eps^2 \right\},
\]
with the convention that $\inf\emptyset{=}{+}\infty$.
Since $0{<}F(0){\le}\eps^2$, we have by Relation~\eqref{e402}, $F'(0){\le}{-}qF(0){<}0$, so that $\tau{>}0$.

The constants $\eps$ and $q$ in Inequality~\eqref{e402} do not depend on the initial point $\Lambda(0)$ as long as $\|\Lambda(0){-}\pi\|_2{\le} \eps$, we can apply the same argument considering the equation starting from $\Lambda(s)$, for $s{\ge}0$ as long as $\|\Lambda(s){-}\pi\|_2{\le} \eps$, and infer the differential inequality with $0$ replaced by $s$. We have therefore the relation 
\[
F'(s)\le  -qF(s).
\]
hence $F(s){\le} F(0)\exp(-qs)$, for all $0{\le}s{<}\tau$. This implies that $\tau$ is infinite.   Inequality~\eqref{e115} is established. It remains to prove our two claims, i.e.,  Relation \eqref{e120} and Inequality \eqref{e401}.

\medskip

\noindent 
   {\sc Step 3: Proof of the identity~\eqref{e120}.}\\
    From now on, $z$ denotes a generic element $(x,y){\in}{\cal X}$ of the state space and $e_1{=}(1,0)$, $e_2{=}(0,1)$ are the unit vectors of ${\cal X}$.

From  Equation~\eqref{DynSys}, we get
\begin{multline}\label{e117}
   \Lambda'(t)(z) = \lambda\left(\rule{0mm}{4mm}\Lambda(t)(z{-}e_2){-}\ind{z{\in}{\cal X}_+}\Lambda(t)(z)\right)\frac{\Lambda(t)({\cal X}_+^c)}{1{-}\Lambda(t)({\cal X}_+^c)}
\\  +\lambda\left(\rule{0mm}{4mm}\Lambda(t)(z{-}e_1)-\ind{z{\in}{\cal X}_+} \Lambda(t)(z)\right)
+\mu_1\left(\rule{0mm}{4mm}(x{+}1)\Lambda(t)(z{+}e_1)-x\Lambda(t)(z)\right)\\+\mu_2\left(\rule{0mm}{4mm}(y{+}1)\Lambda(t)(z{+}e_2)-y\Lambda(t)(z)\right),
\end{multline}
with the convention $\Lambda(t)(z'){=}0$ if $z'{\notin}{\cal X}$.

Similarly, the definition~\eqref{LinDynSys} gives
\begin{multline}\label{e118}
  \Lambda_R'(t)(z) = \lambda\left(\rule{0mm}{4mm}\Lambda_R(t)(z{-}e_2){-}\ind{{\cal X}_+}(z)\Lambda_R(t)(z)\right)\frac{R}{1{-}R}\\
+\lambda\left(\rule{0mm}{4mm}\Lambda_R(t)(z{-}e_1){-}\ind{{\cal X}_+}(z)\Lambda_R(t)(z)\right)+\mu_1\left(\rule{0mm}{4mm}(x+1)\Lambda_R(t)(z{+}e_1){-}x\Lambda_R(t)(z)\right)\\
  +\mu_2\left(\rule{0mm}{4mm}(y+1)\Lambda_R(t)(z{+}e_2){-}y\Lambda_R(t)(z)\right).
\end{multline}
Taking $t{=}0$ in these two relations and using  that $(\Lambda_R(t))$ and $(\Lambda(t))$ have the same initial conditions, we get the identity~\eqref{e120}. We now establish the most intricate inequality of our theorem, namely Inequality~\eqref{e401}.
\medskip

\noindent
    {\sc Step 4: A bound for ${\cal I}_2$.}
From the expression~\eqref{LR} of $L_{\Lambda, R}$, we get
\[
|L_{\Lambda, R}|\le \frac{1}{(1{-}R)^2}\frac{|\Lambda(0)({\cal X}_+^c){-}R|}{1{-}|\Lambda(0)({\cal X}_+^c){-}R|/(1{-}R)},
\]
and, by the Cauchy-Schwartz Inequality, 
\begin{multline*}
|\Lambda(0)({\cal X}_+^c){-}R|=\left|\sum_{z{\in}{\cal X}_+^c}[\Lambda(0){-}\pi](z)\right|= \left|\sum_{z{\in}{\cal X}_+^c}\sqrt{\pi(z)}\frac{[\Lambda(0){-}\pi](z)}{\sqrt{\pi(z)}}\right|\\
\le \left(\sum_{z{\in}{\cal X}_+^c}\pi(z)\right)^{1/2}\left(\sum_{z{\in}{\cal X}_+^c} \frac{([\Lambda(0){-}\pi](z))^2}{\pi(z)}\right)^{1/2}
=\sqrt{R}\|\Lambda(0){-}\pi\|_2=\sqrt{RF(0)}.
\end{multline*}
Combining these two relations, we get the inequality
\begin{align}\label{e123}
|L_{\Lambda, R}|\le \frac{1}{(1{-}R)^2}\frac{\sqrt{R F(0)}}{1 {-} \sqrt{R F(0)}/(1{-}R)}.
\end{align}
Since
\begin{multline}
\sum_{z{\in}{\cal X}} \frac{[\Lambda(0){-}\pi](z)[\Lambda(0)(z{-}e_2)-\ind{z{\in}{\cal X}_+}\Lambda(0)(z)]}{\pi(z)}\\=\sum_{z\in{\cal X},y{\ge}1} \frac{[\Lambda(0){-}\pi](z)[\Lambda(0){-}\pi](z{-}e_2)}{\pi(z)}
{+}\sum_{z{\in}{\cal X}} \frac{[\Lambda(0){-}\pi](z)[\pi(z{-}e_2){-}\ind{z{\in}{\cal X}_+}\pi(z)]}{\pi(z)}\\
{+}\sum_{z{\in}{\cal X}_+} \frac{[\Lambda(0){-}\pi](z)[\pi{-}\Lambda(0)](z)}{\pi(z)}{\steq{def}}{\cal J}_1{+}{\cal J}_2{+}{\cal J}_3,\label{e124}
\end{multline}
we obtain from Relations~\eqref{e123} and~\eqref{e124} that 
\begin{equation}\label{e130}
{\cal I}_2\le 2\lambda |L_{\Lambda,R}|\left(|{\cal J}_1|{+}|{\cal J}_2|{+}|{\cal J}_3|\right).
\end{equation}

\medskip

\noindent
    {\sc Step~5: Estimates for ${\cal J}_i$ and a final bound for ${\cal I}_2$.} \\
Clearly $|{\cal J}_3|{\le} \|\Lambda(0){-}\pi\|_2^2$. By using Relation~\eqref{piR} and the Cauchy-Schwartz Inequality, we get
\begin{multline}\label{e131}
|{\cal J}_1|\le \max_{z\in{\cal X}}{\sqrt{\frac{\pi(z{-}e_2)}{\pi(z)}}}\sum_{z\in{\cal X}, y\ge 1} \frac{|[\Lambda(t){-}\pi](z)|}{\sqrt{\pi(z)}}\frac{|[\Lambda(0){-}\pi](z{-}e_2)|}{\sqrt{\pi(z{-}e_2)}}\\
\le{\sqrt{\frac{C}{R\rho_2}}} \left(\sum_{z{\in}{\cal X}, y{\ge} 1} \frac{([\Lambda(0){-}\pi](z))^2}{\pi(z)}\right)^{1/2}\left(\sum_{z{\in}{\cal X}, y\ge 1} \frac{([\Lambda(0)-\pi](z{-}e_2))^2}{\pi(z{-}e_2)}\right)^{1/2}\\
\le {\sqrt{\frac{C}{R\rho_2}}}\|\Lambda(0){-}\pi\|_2^2={\sqrt{\frac{C}{R\rho_2}}}F(0).
\end{multline}
Another application of the Cauchy-Schwartz Inequality gives
\[
|{\cal J}_2|\le \left(\sum_{z{\in}{\cal X}}\frac{(\pi(z{-}e_2)-\mathbbm{1}_{{\cal X}_+}(z)\pi(z))^2}{{\pi(z)}}\right)^{1/2}\|\Lambda(0){-}\pi\|_2.
\]
On the other hand, by using Relations~\eqref{piR} and~\eqref{PiRFix}, we obtain
 \begin{multline*}
\sum_{z{\in}{\cal X}}\frac{(\pi(z{-}e_2)-\mathbbm{1}_{{\cal X}_+}(z)\pi(z))^2}{{\pi(z)}}\le \sum_{z{\in}{\cal X}}\frac{(\pi(z{-}e_2))^2}{{\pi(z)}}
\\= \sum_{z\in{\cal X}}\pi(z{-}e_2)\frac{y(1{-}R)}{R\rho_2}\le \frac{C(1{-}R)^2}{R\rho_2}
 \end{multline*}
 so that
\[
 |{\cal J}_2|\le (1{-}R)\left(\frac{C F(0)}{R\rho_2}\right)^{1/2}.
\]
 Inequalities~\eqref{e123} and~\eqref{e130} imply
 \begin{multline*}
   {\cal I}_2\le 2\lambda \frac{1}{(1{-}R)^2}\frac{\sqrt{R F(0)}}{1 {-} \sqrt{R F(0)}/(1{-}R)} \\\times\left((1{-}R)\left(\frac{C F(0)}{R\rho_2}\right)^{1/2}
 +\left({\sqrt{\frac{C}{R\rho_2}}}{+}1\right)F(0)\right),
 \end{multline*}
 consequently,
 \[
 \limsup_{\eps{\to}0} \sup_{0{<}F(0){\le} \eps^2} \frac{{\cal I}_2}{F(0)}
\le   \frac{2\lambda}{1{-}R} \sqrt{\frac{C}{\rho_2}}<2\kappa_R,
\]
by Assumption~\eqref{e128}. Relation~\eqref{e401} is thus established. The theorem is proved. 
\end{proof}

\section{The Dynamic Alternative Routing (DAR) Algorithm}\label{DARsec}
  Recall that, for this algorithm,  when a request cannot be accommodated at its arriving node, two other nodes are chosen at random. If both of them are non-saturated, the request takes one place in each of them. Otherwise, the request is rejected. This algorithm has initially been considered to cope with congestion in telephone networks by Gibbens et al.~\cite{Gibbens} in 1990 and in subsequent papers. 

  \subsection{The  Basic ODEs}
\addcontentsline{toc}{section}{\thesubsection \hspace{3.5mm} The  Basic ODEs}

We recall briefly the technical background for this algorithm. See Gibbens et al.~\cite{Gibbens}. There is a set $V$ of vertices and for each couple $(A,B){=}(B,A)$ of vertices, referred to as a link/node, there is a Poisson flow of requests with rate $\lambda$, referred to as calls/jobs, to establish a connection between $A$ and $B$. The capacity constraint is that there are at most $C$ jobs/calls at a given time on any node/link. The state of the process is given by the number of jobs in the links/nodes of the network. The algorithm works as follows. If a node $(A,B)$ has strictly less than $C$ jobs at some instant, then an arriving requests on $(A,B)$ is accepted. Otherwise, a random vertex $C{\not\in}\{A,B\}$ is chosen at random and if both nodes $(A,C)$ and $(C,B)$ have  strictly less than $C$ jobs, then the job occupies a place in $(A,B)$ and in $(C,B)$ during an exponentially distributed amount of time with parameter $1$. If one of the nodes  $(A,B)$, $(C,B)$ is saturated the request is rejected.

The mean-field result of Graham and M\'el\'eard~\cite{GM} described below shows that, from the point of view of the convergence of the empirical distribution process $(\Lambda^N(t))$ defined by Relation~\eqref{EmpIntro}, the DAR algorithm has the same limiting behavior as the following allocation algorithm. There is a set of $N{=}|V|(|V|{-}1)/2$ nodes with finite  capacity $C$, each node receives a Poisson flow of jobs with parameter $\lambda$ to be processed at rate $1$. When a request arrives at a saturated node $\ell$, two other nodes are chosen at random.  If both of them are not saturated, a new request is added to each of them. Otherwise the initial request to node $\ell$ is rejected.

For $1{\leq}\ell{\le}N$, $L_\ell^N(t)$ denotes the number of jobs in node $\ell$ at time $t{\ge}0$. Note that $(L_\ell^N(t),1{\le}\ell{\le}N)$  is not a Markovian process.  The mean-field result of Graham and M\'el\'eard~\cite{GM} conjectured by Gibbens et al.~\cite{Gibbens} is as follows. The initial state is given by i.i.d. random variables with some distribution $\mu$ on $\{0,\ldots,C\}$  for the $L_\ell^N(0)$, $1{\le}\ell{\le}N$, and without any request using two nodes.  It has been shown in~\cite{GM} that the  convergence in distribution, 
\begin{equation}
  \lim_{N\to+\infty} (L_\ell^N(t))= (\overline{L}_C(t)),
\end{equation}
holds for any $\ell{\ge}1$, where, for $t{\ge}0$, the distribution of $\overline{L}_C(t)$, given by the vector $(\P(\overline{L}_C(t){=}k), 0{\le}k{\le}C){=}(x^C_{k}(t),0{\le}k{\le}C)$, is the solution of the following ODEs, for $1{\leq}j{<}C$,
\begin{equation}\label{GHK}
\begin{cases}
\displaystyle   \frac{\diff x^C_{0}}{\diff t}(t) \displaystyle = x^C_{1}(t)-  \lambda h\left(x^C_{C}(t)\right)x^C_{0}(t),\vspace{2mm}\\
\displaystyle  \frac{\diff x^C_{j}}{\diff t}(t) \displaystyle =  \lambda h\left(x^C_{C}(t)\right)x^C_{{j-1}}(t)+ (j{+}1) x^C_{{j+1}}(t)-( \lambda h\left(x^C_{C}(t)\right){+}j )x^C_{j}(t),\vspace{2mm}\\
\displaystyle    \frac{\diff x^C_{C}}{\diff t}(t)\displaystyle =   \lambda h\left(x^C_{C}(t)\right)x^C_{{C-1}}(t)- C x^C_{C}(t),
\end{cases}
\end{equation}
with  $h(x){\steq{def}} (1{+}2 x(1{-}x))$, with initial condition $(x^C_k(0)){=}(\mu(k))$. 

We give an intuitive explanation of this system, to explain the role of the function $h$ in particular. For $t{\ge}0$, $x_C^C(t)$ is the fraction of saturated nodes, i.e, the number of nodes with $C$ jobs. The mean-field limit expresses an asymptotic independence property: the numbers of jobs at a fixed finite subset of nodes are, in the limit, independent and  $x_C^C(t)$ is the probability that an arbitrary node is saturated. Each non-saturated node accepts external requests arriving at rate $\lambda$. It can also be occupied with a request which has arrived at a saturated node and if this request has picked (at random) this empty node and another non-saturated node. With the independence approximation, this occurs with probability
\[
2x_C^C(t)(1-x_C^C(t)).
\]
The equations of the system~\eqref{GHK} can then be easily be explained.

Equations~\eqref{GHK} can be seen as the set of Fokker-Planck equations for a non-homogeneous $M/M/C/C$ queue for which the arrival rate at time $t$ is $\lambda h\left(x^C_{C}(t)\right)$ and the service rate is $1$. Equivalently, from a probabilistic point of view, the process $(\overline{L}_C(t))$ has the same distribution as the solution of the following McKean-Vlasov SDE
\begin{multline}\label{GHKMV}
\diff \overline{L}_C(t)=\ind{\overline{L}_C(s{-}){<}C}{\cal P}_1\left(\left[0,\lambda h\left(\P\left(\overline{L}_C(t){=}C\right)\right)\right]{\times}\diff t\right)\\ {-}{\cal P}_2\left(\left[0,\overline{L}_C(t{-})\right]{\times}\diff t\right),
\end{multline}
with $\overline{L}_C(0){\steq{dist}}\mu$, where ${\cal P}_i$, $i{=}1$, $2$ are independent Poisson processes on $\R_+^2$ with rate $1$. 

An equilibrium point $(x_k^*)$ of the dynamical system~\eqref{GHK} is given by
\[
x_k^*=\frac{1}{Z_C} \frac{\left(\lambda h(x_C^*)\right)^k}{k!}, \quad 0{\le}k{<}C,
\]
where $Z_C$ is the normalization constant and $z{=}x_C^*$ is a positive solution of the fixed point equation
\begin{equation}\label{eqfp}
\frac{\left(\lambda h(z)\right)^C}{C!}-z\sum_{k=0}^C \frac{\left(\lambda h(z)\right)^k}{k!}=0.
\end{equation}
There always exists a solution to this equation since the  left-hand side of Relation~\eqref{eqfp} is positive for $z{=}0$ and negative for $z{=}1$. The rest of this section is devoted to determine the conditions under which there may exist several solutions for this equation and to investigate their stability properties for the dynamical system~\eqref{GHK}. 

The insightful Gibbens et al.~\cite{Gibbens} suggests, through approximations and numerical experiments, that this equation may have in fact several solutions: 
  {\em
\begin{quote}\label{GHKQ}
``Observe the possibility of multiple solutions for $x_C^*$, for C large enough and for a narrow range of the ratio $\lambda/C$. The upper and lower solutions correspond to stable fixed points for the system of equations (2)â(5), while the middle solution corresponds to an unstable fixed point.''
\end{quote}
  }
The notations have been adapted. To the best of our knowledge, these statements do not seem to have been established in a more formal way. The rest of this section is devoted to a scaling analysis of the set of ODEs~\eqref{GHK}. As suggested by these numerical experiments, we will study the case of a large capacity $C$. Concerning the stability results of this assertion, we have not been able to prove them as such. Sections~\ref{2StabSec} and~\ref{nlMM1} give only partial results in this domain.

\subsection{An Asymptotic Dynamical System}\label{ADS}
\addcontentsline{toc}{section}{\thesubsection \hspace{3.5mm} An Asymptotic Dynamical System}
We denote by $(x^C_k(t))$ the solution of the  ODE~\eqref{GHK} when $\lambda$ is replaced by $\lambda C$ and $h(x){=}(1{+}ax(1{-}x))$, for some $a{>}1$ and $x{\in}(0,1)$.
 For this algorithm too,  there is a kind of analogue of the regime analyzed in Section~\ref{sec1-Sat}, in the sense that it has some intuitive explanation. In this regime, in the limit, all jobs are accommodated without rerouting provided that the initial state is not already saturated. As  will be seen, for the same parameters, there are nevertheless two other distinct equilibrium regimes for which a job is rerouted with positive probability as $C$ gets large. 

\begin{theorem}[The Solutions of a Fixed Point Equation]\label{ThNbFP}
 For the fixed point equation~\eqref{eqfp} with $\lambda{=}\nu C$ for some $\nu{>}0$, and  $h(x){=}1{+}ax(1{-}x)$ for $a{>}1$, there exists $C_0{>}0$ such that, for $C{\ge}C_0$,
  \begin{enumerate}
 \item If $\nu{\in}(0,1)$, then there exists a solution $x_{C,1}^*{\in}(0,1)$ of Equation~\eqref{eqfp} such that $\nu h(x_{C,1}^*){<}1$ and
\begin{equation}\label{nuh-1}
    \lim_{C\to+\infty} x_{C,1}^*=0.
\end{equation}
    \item If $\nu{\in}(\nu_a,1)$, with 
\begin{equation}\label{lambdaa}
\nu_a\steq{def}{3}\left/\left(1{+}\frac{2}{9}a{+}\frac{2}{3}\frac{(a{+}3)^{3/2}}{\sqrt{a}}\right),\right.
\end{equation}
then there are three solutions $x_{C,i}^*{\in}(0,1)$, $i{\in}\{1,2,3\}$ of Equation~\eqref{eqfp}, such that $\nu h(x_{C,1}^*){<}1$ and $\nu h(x_{C,i}^*){>}1$, for $i{=}2$, $3$,  and the limiting values of $(x_{C,2}^*)$ and $(x_{C,3}^*)$  are the two  solutions of the polynomial equation
\begin{equation}\label{polfp}
a{z}^{3}{-}2a{z}^{2}{+}(a{-}1)z{+}1=\frac{1}{\nu}
\end{equation}
in $(0,1)$ and $(x_{C,1}^*)$ satisfies Relation~\eqref{nuh-1}.
\item  If $\nu{>}1$, there exists a unique solution $x_{C,1}^*{\in}(0,1)$ of Equation~\eqref{eqfp}, and the sequence $(x_{C,1}^*)$ is converging to the unique solution of Relation~\eqref{polfp} in the interval $(0,1)$. 
  \end{enumerate}
\end{theorem}

For the Gibbens et al.  model which corresponds to the case $a{=}2$, this gives $\nu_2{\sim}0.937$,  hence $(0.937,1)$ is the ``{\em  narrow range of the ratio $\lambda/C$}'' quoted by these authors, see page~\pageref{GHKQ}, for which there are three solutions to the fixed point equation.

According to~(1), when $\nu{<}1$,  there is an equilibrium in the light-load regime ($x_{C,1}^*{\sim}0$). We will see a more precise result, Proposition~\ref{Stab2prop}, concerning the asymptotic local stability of this equilibrium.  When $\nu_a{<}\nu{<}1$, there are two other equilibrium points but in a saturated regime, $x_{C,i}^*{\geq}\eta{>}0$, for $i{=}1$, $2$ and $C$ sufficiently large.
\begin{proof}
The fixed point equation~\eqref{eqfp} with $\lambda$ replaced by $\nu C$ can be rewritten as $\Psi_{C}(z){=}0$, with, for $z{>}0$,
\[
\Psi_C(z)\steq{def}1{-}z\sum_{k=0}^C \frac{C!}{(C{-}k)!}\frac{1}{\left(\nu C h(z)\right)^k}
=1{-}z\sum_{k=0}^C \prod_{i=0}^{k{-}1}\left(1{-}\frac{i}{C}\right) \frac{1}{\left(\nu h(z)\right)^k}.
\]
Note that the function $C{\mapsto}\Psi_C(z)$ is decreasing. 

If $\nu{<}1$, we can choose $\eps{<}1/2$ sufficiently small so that  $\delta{=}\nu h(\eps){<}1/2$ holds.  For $z{\le}\eps$, we have 
\[
\sum_{k=0}^C \prod_{i=0}^{k{-}1}\left(1{-}\frac{i}{C}\right) \frac{1}{\left(\nu h(z)\right))^k}
\ge
\sum_{k=0}^{\lfloor \delta C\rfloor} \prod_{i=0}^{k{-}1}\left(1{-}\frac{i}{C}\right) \frac{1}{\delta^k}
\ge
\sum_{k=0}^{\lfloor \delta C\rfloor} \left(\frac{1{-}\delta}{\delta}\right)^k.
\]
In particular, for $0{<}z{<}\eps$, $(\Psi_C(z))$ converges to ${-}\infty$ as $C$ goes to infinity. Since $\Psi_C(0){=}1$, one can find $C_0{>}0$ such that if $C{\ge}C_0$ there is a zero of $\Psi_C$ in the interval $(0,\eps)$.

For $\delta{>}1$, it is easily checked that the convergence
\begin{equation}\label{cvC}
\lim_{C\to+\infty} \Psi_{C}(z) =\overline{\Psi}(z)\steq{def}1{-}z\frac{1}{1-1/(\nu h(z))}=\frac{\nu(1{-}z)h(z){-}1}{\nu h(z){-}1}
\end{equation}
holds uniformly for all  $z{\in}(0,1)$  such that  $\nu h(z){>}\delta$. Equation~\eqref{eqfp} becomes, in the limit,
\begin{equation}\label{eq2-2}
(1{-}z)h(z)=\frac{1}{\nu}.
\end{equation}
Note that any solution $z{<}1$ of such equation satisfies $\nu h(z){>}1$.

The quantity $(1{-}z)h(z)$ is the polynomial $a{z}^{3}{-}2a{z}^{2}{+}az{-}z{+}1$ which is increasing from $1$ on the interval $[0,x_0]$ and decreasing on $[x_0,1]$, with
\[
x_0\steq{def}\frac{1}{3}\left(2{-}\sqrt{\frac{a{+}3}{a}}\right),\quad  (1{-}x_0)h(x_0){=}1{/}\nu_a{=}\frac{1}{3}\left(1{+}\frac{2}{9}a{+}\frac{2}{3}\frac{(a{+}3)^{3/2}}{\sqrt{a}}\right).
\]
Equation~\eqref{eq2-2} has therefore two solutions if and only if  $\nu{\in}(\nu_a,1)$, one solution when $\nu{=}\nu_a$ or $\nu{>}1$ and none if $\nu{<}\nu_a$.

If $\nu{\in}(\nu_a,1)$, for $\eps{>}0$ sufficiently small, there exist $z_0{<}z_1{<}z_2$ such that $\nu h(z_i){>}1$, for $i{\in}\{0,1,2\}$, and  $\overline{\Psi}(z_0){<}{-}\eps$, $\overline{\Psi}(z_1){>}\eps$ and $\overline{\Psi}(z_2){<}{-}\eps$. Consequently, there exists $K_0$, such that, if $C{\ge}K_0$, then the last three inequalities hold with $\overline{\Psi}$ replaced by $\Psi_C$ and $\eps$ by $\eps/2$.  Hence, we get that there are two solutions of the relation $\Psi_C(z){=}0$ such that $\nu h(z){>}1$ for $C{\ge}K_0$. Assertion~(2) is proved.

If $\nu{>}1$, the convergence~\eqref{cvC} is uniform for  $z{\in}[0,1]$. Since $\overline{\Psi}(0){>}0$ and  $\overline{\Psi}(1){<}0$,  by the same argument as before, there exists some $K_1$ such that if $C{\ge}K_1$ then there is a solution $x_{C,1}^*$ of Relation~\eqref{eqfp}. A simple calculation gives
\[
\frac{\diff}{\diff z} \overline{\Psi}(z)=-\frac{\nu}{(\nu h(z){-}1)^2}\left(\nu{z}^{2} \left( z{-}1 \right) ^{2}{a}^{2}{-}{z}\left( 2\nu (z{-}1){-}3z{+}2 \right) a{+}\nu{-}1\right),
\]
and one has
\begin{multline*}
  \nu{z}^{2}\left( z{-}1 \right) ^{2}{a}^{2}{-}{z}\left( 2\nu(z{-}1){-}3z{+}2 \right) a\geq az (3z{-}2+ \left( {z}^{3}{-}2{z}^{2}{-}z{+}2 \right) \nu)\\
\ge  az (3z{-}2+ \left( {z}^{3}{-}2{z}^{2}{-}z{+}2 \right))=az^2(z^2{-}2z{+}2)\ge 0,
\end{multline*}
by using successively  that $a{\ge}1$, and that ${z}^{3}{-}2{z}^{2}{-}z{+}2{=}(1{-}z)(2{+}2z{-}z^2){\ge}0$ when $z{\in}[0,1]$, and finally  $\nu{\ge}1$. 
Consequently, we get that
\begin{equation}\label{eqts}
\frac{\diff}{\diff z} \overline{\Psi}(z)\leq-\frac{\nu(\nu{-}1)}{(\nu h(z){-}1)^2}\leq -\frac{\nu(\nu{-}1)}{(\nu (1{+}a/4){-}1)^2}<0
\end{equation}
holds for all $z{\in}[0,1]$.

Note that, as for Relation~\eqref{cvC}, since $\nu{>}1$, the convergence
\[
\lim_{C{\to}{+}\infty} \frac{\diff}{\diff z} {\Psi}_C(z)=\frac{\diff}{\diff z} \overline{\Psi}(z)
\]
holds uniformly for $z{\in}[0,1]$. 

To prove the uniqueness, we assume that  there is a sub-sequence $(C_n)$ converging to infinity for which  the  equation $\Psi_{C_n}(z){=}0$ has two solutions. It implies in particular that we have a sequence $(z_n)$ of $(0,1)$ such that ${\Psi}_{C_n}'(z_n){=}0$. Due to the uniform convergence, this is in contradiction with Relation~\eqref{eqts}.  The proposition is proved.

\end{proof}
When $\nu{<}1$,  with our method based on the asymptotic behavior of $\Psi_C(z)$ as $C$ gets large, we have not been able to prove that {\em all} solutions of Equation~\eqref{eqfp} are identified though this is very likely the case.

\subsection*{A Scaled Version of the Dynamical System}
For the moment we have given a scaled version of the fixed point equations. It turns out that one can also get some insight from a scaled version of the dynamical system~\eqref{GHK} converging to a non-trivial dynamical system whose fixed points are described in Theorem~\ref{ThNbFP}.

Let us introduce some notations.  The set of bounded sequences is denoted by ${\cal B}(\N)$, it is endowed with the norm, for $z{\in}{\cal B}(\N))$,
\[
\|z\|=\sum_{k=0}^{+\infty}\frac{1}{2^k} |z_k|.
\]
For $T{>}0$ and $(z(t)){=}(z_k(t)){\in}{\cal C}(\R_+,{\cal B}(\N))$  a continuous function on $\R_+$, we define
\[
\|z\|_{T}\steq{def} \sup_{0{\le}t{\le}T} \|z(t)\|.
\]
Additionally ${\cal P}(\N)$ is the set of probability distributions on $\N$.

The scaling consists of slowing down the time scale by a factor $C$ and by looking at the number of empty places for the McKean-Vlasov process. 
\begin{proposition}[Asymptotic Dynamical System]\label{propscale}
  If $(x_k^C(t))$ is the solution of the set of ODEs  defined by Relation~\eqref{GHK} with $\lambda{=}\nu C$, with an initial point such that
\[
\lim_{C\to+\infty}\|(x^C_{C{-}k}(0), k{\in}\N){-}(q_0(k),k{\in}\N)\|=0
\]
for some probability distribution $q_0{\in}{\cal P}(\N)$ then, as $C$ goes to infinity,  the process
  \[
  (y_k^C(t),k{\in}\N)\steq{def} \left(x^C_{C-k}\left({t}/{C}\right), k{\in}\N\right)
  \]
 is converging in distribution for the uniform  norm $\|{\cdot}\|_T$ to $(\Gamma(t)){\in}{\cal C}(\R_+,{\cal P}(\N))$ which is the unique solution of the set of differential equations
 \begin{equation}\label{MM1NL}
\begin{cases}
\displaystyle    \frac{\diff}{\diff t} \Gamma_{0}(t) =   \nu h\left(\Gamma_{0}(t)\right)\Gamma_{{1}}(t)-  \Gamma_{0}(t),\vspace{2mm}\\
\displaystyle  \frac{\diff}{\diff t} \Gamma_{k}(t) {=}  \nu h\left(\Gamma_{0}(t)\right)\Gamma_{{k+1}}(t){+} \Gamma_{{k-1}}(t){-}\left( \nu h\left(\Gamma_{0}(t)\right){+}1\right)\Gamma_{k}(t),\quad k{\ge}1,
\end{cases}
 \end{equation}
 with $\Gamma(0){=}q_0$.
\end{proposition}
\begin{proof}
  Note that the process $(y_k^C(t))$ can be seen as a version of the empirical distribution process of empty places in the nodes of the network with the slowed down time scale $t{\mapsto}t/C$.

It is not difficult to check that $(y_k^C(t))$ satisfies the following system of ODEs
\begin{equation}\label{GHK3}
\begin{cases}
\displaystyle    \frac{\diff}{\diff t} y^C_{0}(t) =   \nu h\left(y^C_{0}(t)\right)y^C_{{1}}(t)-  y^C_{0}(t),\vspace{2mm}\\
\displaystyle  \frac{\diff}{\diff t} y^C_{k}(t) {=}  \nu h\left(y^C_{0}(t)\right)y^C_{{k+1}}(t){+} \left(1{-}\frac{k{+}1}{C}\right) y^C_{{k-1}}(t)\\
\hspace{55mm}  \displaystyle {-}\left( \nu h\left(y^C_{0}(t)\right){+}\left(1{-}\frac{k}{C}\right)\right)y^C_{k}(t),\vspace{2mm}\\
  \displaystyle   \frac{\diff}{\diff t} y^C_{C}(t)  = \frac{1}{C}y^C_{1}(t)-  \nu h\left(y^C_{0}(t)\right)y^C_{C}(t).
\end{cases}
\end{equation}
The proposition  is proved with a classical  consequence of compactness-uniqueness argument: the sequence of functions $(y_k^C(t))$, $C{\ge}1$ is tight and any limiting point satisfies Relations~\eqref{GHK3}. If this system of equations has a unique solution, then the convergence is established since all subsequences have a subsequence converging to this limit. See Chapter~10 and~11 of Ethier and Kurtz~\cite{Ethier} for example.

The tightness is due to the Arzel\`a-Ascoli Theorem, see Theorem~7.2 of Billingsley~\cite{Billingsley} for example. Relations~\eqref{GHK3} show that for any $k{\in}\N$, the sequence of functions $(y^C_k(t))$ is equicontinuous and therefore is relatively compact for the uniform norm on bounded intervals. The integral form of the ODEs~\eqref{GHK3} shows that any limiting point satisfies Relation~\eqref{MM1NL}.

Let $(x_k(t))$ and $(y_k(t))$ be two solutions of Relation~\eqref{MM1NL} with the same initial condition. Using the fact that the function $h$ is lipshitz on $[0,1]$ with parameter $6$,  the integral form of Relation~\eqref{MM1NL} gives the inequality, for $(u_k(t)){=}(x_k(t){-}y_k(t))$ and $t{>}0$,
\begin{multline*}
|u_k(t)|\leq 6\nu\int_0^t |u_0(s)|(x_k(s){+}x_{k+1}(s))\,\diff s \\+6\nu \int_0^t  (|u_{k-1}(s)|+|u_{k}(s)|{+}|u_{k+1}(s)|)\,\diff s,
\end{multline*}
since $(x_k(t))$ and $(y_k(t))$ are probability distributions on $\N$, then $(u_k(t))$ is a convergent series and
\[
\sum_{k=0}^n |u_k(t)|\leq 30\nu \int_0^t  \sum_{k=0}^{n+1} |u_k(t)|\,\diff s,
\]
so that
\[
U(t){\steq{def}}\sum_{k=0}^{+\infty} |u_k(t)|\leq 30\nu \int_0^t  U(s)\,\diff s.
\]
Gr\"onwald's Inequality gives the relation $(U(t)){=}(0)$, i.e.\  the uniqueness of the solution of Equation~\eqref{MM1NL}, and, consequently, the desired convergence. 

\end{proof}

A probabilistic translation of this result can be stated  as the fact that if the process  $(\overline{L}_C(t))$, defined by Relation~\eqref{GHKMV} satisfies the relation 
\[
\lim_{C\to+\infty} C{-}\overline{L}_C(0)=q_0,
\]
for the convergence in distribution, then, 
\[
\lim_{N\to+\infty} (Q_C(t)){\steq{def}}(C{-}\overline{L}_C(t/C))=(\overline{Q}(t)),
\]
where  $(\overline{Q}(t))$ is the  solution of the McKean-Vlasov SDE
\begin{equation}\label{MM1MV}
  \diff \overline{Q}(t)={\cal N}_1\left(\diff t\right)-\ind{\overline{Q}(s{-}){>}0}{\cal P}_1\left(\left[0,\nu h\left(\P\left(\overline{Q}(t){=}0\right)\right)\right]{\times}\diff t\right),
\end{equation}
with $\overline{Q}(0){\steq{dist}}q_0$, where ${\cal N}_1$ and ${\cal P}_1$ are independent homogeneous Poisson processes with rate  $1$ on $\R_+$ and $\R_+^2$ respectively. 

The process $(\overline{Q}(t))$ is a non-linear $M/M/1$ queue with the jump rates at time $t$
\begin{equation}\label{MM1JR}
\begin{cases}
  {+}1& 1\\
  {-}1& \nu h\left(\P\left(\overline{Q}(t){=}0\right)\right).
\end{cases}
\end{equation}
It should be noted that this scaling is convenient to study the  regimes where the number of empty places is small, i.e.\ when the system has some level of saturation.  Section~\ref{2StabSec} studies the case when there is an equilibrium regime with a large number of empty places. 
\begin{remark}
An invariant distribution $\overline{\pi}$ of the non-linear Markov process $(\overline{Q}(t))$  defined by Relation~\eqref{MM1MV} is the invariant distribution of an $M/M/1$ queue with arrival rate $1$ and service rate $\nu h(\overline{\pi}(0))$, $\overline{\pi}$ is thus a geometric distribution with parameter $1/(\nu h(\overline{\pi}(0)))$, in particular
  \[
  \overline{\pi}(0)=1{-}\frac{1}{\nu h(\overline{\pi}(0))},
  \]
  which is Relation~\eqref{eq2-2}, as can be  expected.
\end{remark}
\subsection{Stability of the Underloaded Regime}\label{2StabSec}
\addcontentsline{toc}{section}{\thesubsection \hspace{3.5mm} Stability of the Underloaded Regime}
In Theorem~\ref{ThNbFP}, we have seen that, under the condition  $\nu{<}1$, there is a root of the fixed point equation~\eqref{eqfp} that is arbitrarily close to $0$ as $C$ gets large. This result suggests that the stability of the underloaded regime for the dynamical system~\eqref{GHK}. In this regime  most of requests are accepted at their arrival node. The following proposition gives a formal characterization of this property.
\begin{proposition}[Stability of Underloaded Regime]\label{Stab2prop}
 If $\lambda{=}\nu C$, for some $\nu{<}1$, there exists $\eta{\in}(0,1)$ such that if the initial state of Dynamical System~\eqref{GHK} satisfies the relation
  \[
  \lim_{C\to{+}\infty} \sum_{k{\ge}\eta C}x_k^C(0)=0,
  \]
then there exists $\eta^*{\in}(0,1)$ such that
  \begin{equation}\label{StabUnder}
  \lim_{C{\to}{+}\infty} \sup_{t{\ge}0} \left(\sum_{k{\ge}\eta^* C} x_k^C(t)\right)=0.
  \end{equation}
\end{proposition}
\begin{proof}
We fix $\delta_0{>}0$ such that $\eta_0{=}\nu(1{+}a\delta_0){\in}(\eta,1)$, and $\eta_1$, $\eta_2{\in}(\eta_0,1)$ with $\eta_1{<}\eta_2$.
For $\eps{<}\delta_0/2$, we take $C_0$, such that
  \[
  \sum_{k{\geq} \eta_0C_0} x_k^C(0)\leq \eps \text{ and } \left(\frac{\eta_0}{\eta_1}\right)^{\lfloor (\eta_2{-}\eta_1)C_0\rfloor}\leq \eps.
  \]
Let $(Q(t))$ be an $M/M/1$ queue with arrival rate $\eta_0$ and service rate $\eta_1$ with $Q(0){=}0$ and $(\overline{L}_C(t))$ the processes defined by Equation~\eqref{GHKMV}. By stochastic monotonicity of $t{\mapsto}Q(t)$, which can be seen with a simple coupling of $(Q(t))$ with a stationary version of the process $(Q(t))$. See the proof of Proposition~5.8 of Robert~\cite{Robert} for example. 
We have, for $C{\ge}C_0$,
\[
\P\left(Q(t) \geq (\eta_2{-}\eta_1)C\right)\\ \leq \P(Q(\infty)\geq  (\eta_2{-}\eta_1)C) \leq \left(\frac{\eta_0}{\eta_1}\right)^{\lfloor (\eta_2{-}\eta_1)C\rfloor}\leq \eps,
\]
where $Q(\infty)$ is a geometrically distributed random variable with parameter $\eta_0/\eta_1$. If $(Q_1(t)){\steq{def}}\eta_1 C{+} Q(t)$, then, for all $t{\ge}0$,
\[
\P(Q_1(t)\geq \eta_2 C)\leq \eps\leq \frac{\delta_0}{2}.
\]
Note that $(Q_1(C t))$ is a birth and death Markov process with birth rate $\eta_0 C$ and death rate $\eta_1 C$.

We can construct a coupling of the process $(\overline{L}_C(t))$ and $(Q(t))$ such that, for all $t{\ge}0$, the relation
\[
\overline{L}_C(t)\le Q_1(Ct)
\]
holds conditionally on the event $\{\overline{L}_C(0){<} \eta_0 C\}$.   This is a simple consequence of the fact that, as long as the relation $x^C_C(t){=}\P(\overline{L}_C(t){=}C){\le} \delta_0$ holds,  the input rate of $(\overline{L}_C(t))$ is smaller than $\eta_0 C$ and when $\overline{L}_C(t){>}\eta_1 C$, the departure rate is at least $\eta_1 C$.
We obtain that, for $C{\ge}C_0$ and all $t{\ge}0$,
\begin{multline*}
\sum_{k{\ge}\eta_2C} x_k^C(t){=}\P\left(\overline{L}_C(t){\ge}\eta_2C\right)\le  \P\left(\overline{L}_C(t){\ge}\eta_2 C|\overline{L}_C(0){\le} \eta_0 C\right){+}\eps\\ \leq  \P\left(Q_1(Ct) \geq \eta_2 C\right){+}\eps\leq 2\eps.
\end{multline*}
The proposition is proved.
\end{proof}
\subsection{Non-Linear $M/M/1$ Queues}\label{nlMM1}
\addcontentsline{toc}{section}{\thesubsection \hspace{3.5mm} Non-Linear $M/M/1$ Queues}

The  Relations~\eqref{MM1NL} defining the asymptotic process $(\Gamma(t))=(\Gamma_k(t),k{\in}\N)$ can be written in a vectorial form as, for any function $f{:}\N{\to}\R_+$ with finite support, 
\begin{equation}\label{MNL}
\frac{\diff }{\diff t} \croc{\Gamma(t),f} =\croc{\Gamma(t),\nabla^+(f)}  + \nu h(\Gamma(t)(0))\croc{\Gamma(t),\nabla^-(f)},
\end{equation}
with
\[
\nabla^+(f)(x)\steq{def}f(x{+}1){-}f(x), \quad \nabla^-(f)(x)\steq{def}\left(f(x{-}1){-}f(x)\right)\ind{x{>}0}.
\]
In Relations~\eqref{MM1NL}, we had $h(x){=}1{+}ax(1{-}x)$ with $a{>}1$.

For the moment we consider a general function $h$ which is  continuously differentiable  from $[0,1]$ to $[1,{+}\infty)$, and we assume that $\nu{>}1$. If there is an equilibrium $\pi{\in}{\cal P}(\N)$ for  the dynamical system~\eqref{MNL}, it is the equilibrium of the linear Markov process $(\Gamma_S(t))$, where $S{=}\pi(0)$, 
\begin{equation}\label{ML}
\frac{\diff }{\diff t} \croc{\Gamma_S(t),f} =\croc{\Gamma_S(t),\nabla^+(f)}  + \nu h(S)\croc{\Gamma_S(t),\nabla^-(f)}.
\end{equation}
As noted before, this is a classical $M/M/1$ queue which is ergodic since the service rate $\nu h(S)$ is greater than $1$, the arrival rate, by assumption. We have therefore a representation for the invariant distribution
\begin{equation}\label{2pir}
  \pi_{S}(n)\steq{def}\left(\frac{1}{\nu h(S)}\right)^n\left(1{-}\frac{1}{\nu h(S)}\right), \quad n{\in}\N,
\end{equation}
and, consequently, the fixed point equation for $S$,
\begin{equation}\label{2PirFix}
S=1{-}\frac{1}{\nu h(S)},
\end{equation}
which we have already seen, see Equation~\eqref{polfp}.  It is well-known that, for the  standard $M/M/1$ process $(\Gamma_S(t))$,   for any initial condition $\Gamma_S(0)\in{\cal P}(\N)$, the inequality
 \begin{equation}\label{e308}
 \frac{\diff}{\diff t}\|\Gamma_S(t){-}\pi_S\|_2^2\le {-}2\kappa_S \|\Gamma_S(t){-}\pi_S\|^2
 \end{equation}
holds for all $t{\ge}0$, where $\kappa_S$, the spectral gap of the process, has the explicit representation
\[
 \kappa_S =\left(\sqrt{\nu h(S)}{-}1\right)^2. 
 \]
 See Chen~\cite{Chen} and Liu and Ma~\cite{Liu} for example. 

 \begin{theorem}\label{theoNLMM1}
If $h{:}[0,1]{\to}[1,{+}\infty)$ is a $C^1${-}function and $S{\in}(0,1)$ is  a solution of Equation~\eqref{2PirFix}   such that
\begin{equation}\label{e310}
|\dot{h}(S)|< \frac{1}{\nu S}\left(\frac{1}{\sqrt{1{-}S}}{-}1\right)^2,
 \end{equation} 
then the probability distribution  $\pi_S$ defined by Relation~\eqref{2pir} is an exponentially stable  equilibrium point of the dynamical system defined by~\eqref{MNL}: There exist positive constants $q$ and $\eps$ such that if $\|\Gamma(0){-}\pi_S\|_2\le\eps$, the relation
\[
\|\Gamma(t){-}\pi_S\|_2^2\le \|\Gamma(0){-}\pi_S\|_2^2\cdot e^{{-}q t},  
\]
holds for all $t{\ge}0$.
 \end{theorem}
\begin{proof}
Identity~\eqref{MNL} gives
\[
\dot{\Gamma}(t)(x) = \left(\rule{0mm}{4mm}\Gamma(t)(x{-}1){-}\Gamma(t)(x)\right)\ind{x{>}0}{+} \nu h(\Gamma(t)(0))\left(\rule{0mm}{4mm}\Gamma(t)(x{+}1){-}\Gamma(t)(x)\right),
\]
in the same way, with Relation~\eqref{ML},  $({\Gamma}_S(t))$  is defined as the solution of 
\[
\dot{\Gamma}_S(t)(x) = \left(\rule{0mm}{4mm}\Gamma_S(t)(x{-}1){-}\Gamma_S(t)(x)\right)\ind{x{>}0}{+} \nu h(S)\left(\rule{0mm}{4mm}\Gamma_S(t)(x{+}1){-}\Gamma_S(t)(x)\right),
\]
with the same initial conditions $\Gamma_S(0){=}\Gamma(0)$. We have
\[
\dot{\Gamma}(0)(x)=\dot{\Gamma}_S(0)(x) + \nu(h(\Gamma(0)(0)){-}h(S))\left(\rule{0mm}{4mm}\Gamma(0)(x{+}1){-}\Gamma(0)(x)\right).
\]
As in Section~\ref{sec1-Spec}, introducing
\[
F(t){\steq{def}}\|\Gamma(t){-}\pi_S\|_2^2,
\]
we have
\begin{multline}\label{e314}
\dot{F}(0)=2\sum_{x{\in}\N} \frac{[\Gamma_S(0){-}\pi_S](x)\dot{\Gamma}_S(0)(x)}{\pi_S(x)}\\
+2\nu (h(\Gamma(0)(0)){-}h(S))\sum_{x\in\N} \frac{[\Gamma(0){-}\pi_S](x)(\Gamma(0)(x{+}1){-}\Gamma(0)(x))}{\pi_S(x)}\\
\steq{def}{\cal I}_1+2\nu (h(\Gamma(0)(0)){-}h(S))\cdot {\cal I}_2.
\end{multline}
By Relation~\eqref{e308} and the fact that $(\Gamma_S(t))$ and $(\Gamma(t))$ have the same initial condition, we get 
\begin{equation}\label{e313}
\frac{{\cal I}_1}{2}\le {-}2\kappa_S F(0)={-}2\left(\sqrt{\nu h(S)}{-}1\right)^2F(0).
\end{equation}
Furthermore, we have
\begin{align*}
  {\cal I}_2&\steq{def}\sum_{x\in\N}\frac{[\Gamma(0){-}\pi_S](x)(\Gamma(0)(x{+}1){-}\Gamma(0)(x))}{\pi_S(x)}\\
  &=\sum_{x\in\N}\frac{[\Gamma(0){-}\pi_S](x)(\pi_S(x{+}1){-}\pi_S(x))}{\pi_S(x)}\\
&\hspace{3cm} + \sum_{x\in\N}\frac{[\Gamma(0){-}\pi_S](x)[\Gamma(0){-}\pi_S](x{+}1)}{\pi_S(x)}{-}\|\Gamma(0){-}\pi_S\|_2^2.
\end{align*}
Cauchy-Schwartz Inequality and Relation~\eqref{2pir} give
\begin{align*}
  |{\cal I}_2|&\le \left(\sum_{x\in\N} \frac{(\pi_S(x{+}1){-}\pi_S(x))^2}{\pi_S(x)}\right)^{1/2}\hspace{-5mm}\|\Gamma(0){-}\pi_S\|_2 \\
  &\hspace{3cm}+ \sup_{x\in\N}\left(\sqrt \frac{\pi_S(x{+}1)}{\pi_S(x)}{+}1\right)\|\Gamma(0){-}\pi_S\|_2^2 \\
&=\left(1{-}\frac{1}{\nu h(S)}\right)\|\Gamma(0){-}\pi_S\|_2+ \left(\frac{1}{\sqrt{\nu h(S)}}{+}1\right)\|\Gamma(0){-}\pi_S\|_2^2.
\end{align*}
Combining this  with Relation~\eqref{2PirFix} and  Inequalities~\eqref{e314} and~\eqref{e313} and the fact that $|\Gamma(0)(0){-}S|{\le}\|\Gamma(0){-}\pi_S\|_2$, we obtain that, if $F(0){>}0$,
\[
\frac{\dot{F}(0)}{2F(0)}
\leq {-}\kappa_S +\nu \frac{|h(\Gamma(0)(0)){-}h(S)|}{|\Gamma(0)(0){-}S|}\left(S+ \left(\frac{1}{\sqrt{\nu h(S)}}{+}1\right)\sqrt{F(0)}\right),
\]
hence
\[
\limsup_{\eps\to 0} \sup_{F(0){\in}(0,\eps^2)} \frac{\dot{F}(0)}{F(0)}\leq -q
=2\left({-}\left(\sqrt{\nu h(S)}{-}1\right)^2+\nu |h'(S)|S\right)<0,
\]
by Assumption~\eqref{e310}. To complete the proof, it remains to apply a similar  bootstrap argument as in the proof of Theorem~\ref{t138} to get the above inequality for the ratio $\dot{F}(s)/F(s)$. The theorem is proved. 
 \end{proof}
We now apply this result to the asymptotic dynamical system of the DAR algorithm. 
\begin{corollary}
  When $h(x) {=} 1{+}ax(1{-}x)$ with $a{>}1$, there exists a neighborhood $I$ of $u_a{\steq{def}}8/(4{+}a)$ such that, if $\nu{\in}I$, then the unique fixed point of the dynamical system~\eqref{MM1NL} of the non-linear $M/M/1$ queue is  exponentially stable.
\end{corollary}
\begin{proof}
The fixed point equation~\eqref{2PirFix} is 
\begin{equation}\label{epf}
(1{-}S)(1{+}aS (1{-}S))=\frac{1}{\nu}
\end{equation}
and since $h'(x){=}a\nu(1{-}2x)$,  Condition~\eqref{e310} is equivalent to 
\begin{equation}\label{epf2}
  aS |1{-}2S| < \left(1{-}\sqrt{1{-}S}\right)^2  (1{+}aS(1{-}S)).
\end{equation}
To conclude, note that $S{=}1/2$ satisfies the condition.
\end{proof}
\noindent
{\bf Remarks.}
\begin{enumerate}
\item It easy to check that, for $a{>}4$, $u_a{\in}(\nu_a,1)$, where $\nu_a$ is defined by Relation~\eqref{lambdaa}. In this case there are two positive fixed points for the asymptotic dynamical system, the above corollary  gives that one of them is locally stable.  We have not been able to prove that, as Gibbens et al.~\cite{Gibbens} suggest, see the claim page~\pageref{GHKQ}, that the other one is not stable. 
\item A little more  work can give more precise conditions on $\nu$ for the stability of the fixed point. 
Let $x{=}\sqrt{1{-}S}$, if $S{\in}[0,1/2]$, Condition~\eqref{epf2} amounts to 
\[
P_1(x)\steq{def}ax^{5}{-}ax^{4}{-}3a{x}^3 {-}ax^{2}{+} (a{-}1) x{+}a{+}1>0.
\]
Notice that  $P_1(\sqrt{2}/2){=}(2{-}\sqrt{2})(1{+}a/4)/2$ and $P_1(1){=}{-}2a$.
If $S{\in}[1/2,1]$,  the condition is 
\[
P_2(x)\steq{def}ax^5{-}a{x}^{4}{+}a{x}^{3}{+}3a{x}^{2}{-}(1{+}a)x{+}1{-}a>0,
\]
with  $P_2(0){=}1{-}a{<}0$ and $P_2(\sqrt{2}/2){=}P_1(\sqrt{2}/2){>}0$. It is not difficult to check that $P_1$ [resp. $P_2$] is concave  [resp. convex] on $[0,1]$, hence there exists a unique root $z_{a,1}$ of $P_1$ in $(\sqrt{2}/2,1)$ [resp.  $z_{a,2}$ of $P_2$ in $(0,\sqrt{2}/2)$].

Hence Condition~\eqref{e310} is satisfied when $S{\in}(1{-}z_{a,1}^2,1{-}z_{a,2}^2)$ and, by Relation~\eqref{epf}, this holds if $\nu{\in}(Q(z_{a,2}),Q(z_{a,1})$ with 
$Q(z){=}1/[z^2(1{+}az^2{-}az^4)]$.
\item For the precise case of Gibbens, Hunt and Kelly~\cite{Gibbens}, $h(x){=}(1{+}2x(1{-}x))$, this gives that when  $\nu{\in}(1.2068,1.5978)$, the unique fixed point  is a locally stable equilibrium.

\end{enumerate}

\medskip

\noindent
{\bf A toy example with an arbitrary number $n$ of stable equilibrium points.}\\  We fix $(u_k)$, $n$ distinct points of $(0,1)$. Let $f$ be a $C^1$-function such that, for $1{\le}k{\le}n$, the relation $f(u){=} 1 {+} \ln(1{-}u) {-} \ln(1{-}u_k)$ holds in a small neighborhood of $u_k$ for any $k\in{1, \ldots, n}$. Note that since $f(u_k){=}1{>}1{-}u_k$,  $1{\le}k{\le}n$, we can choose $f$ in such a way that $f(u){>}1{-}u$ holds for all $u{\in}(0,1)$. If we define $h(u){=}u/(1{-}u)$, for $u{\in}(0,1)$ and $\nu{=}1$, then $h$ maps $(0,1)$ to $(1,{+}\infty)$. Each  $u_k$ ,  $1{\le}k{\le}n$, is clearly a fixed point and Condition~\eqref{e310} is satisfied since $h'(u_k){=}0$. It is therefore locally stable.

\providecommand{\bysame}{\leavevmode\hbox to3em{\hrulefill}\thinspace}
\providecommand{\MR}{\relax\ifhmode\unskip\space\fi MR }
\providecommand{\MRhref}[2]{%
  \href{http://www.ams.org/mathscinet-getitem?mr=#1}{#2}
}
\providecommand{\href}[2]{#2}

\end{document}